\documentclass[11pt,a4paper]{article}
\usepackage[qm]{qcircuit}

\usepackage{mathtools}
\usepackage{authblk} 
\usepackage{algpseudocode,algorithmicx,algorithm}

\usepackage{mathrsfs}
\usepackage{latexsym,bm}
\usepackage{amsmath,amsfonts,amssymb,amsthm}
\usepackage{extarrows}

\usepackage{graphicx,subfigure,epstopdf,float}
\usepackage{enumerate,cases,multirow}
\usepackage{makecell}
\usepackage{caption}
\usepackage{xcolor}

\usepackage{longtable,colortbl,arydshln,threeparttable}
\definecolor{mygray}{gray}{.9}

\usepackage{indentfirst}
\usepackage[top=25mm,bottom=20mm,left=25mm,right=20mm]{geometry}
\renewcommand{\baselinestretch}{1.3}

\usepackage{cite}

\usepackage{listings}

\usepackage{makeidx}        
\usepackage{booktabs}
\usepackage[bookmarks,bookmarksnumbered,colorlinks,citecolor=red,linkcolor=blue,hyperindex,linktocpage=true]{hyperref}

\newcommand{\ket}[1]{| #1 \rangle} 
\newcommand{\bra}[1]{\langle #1 |} 

\def \d {\mathrm{d}}

\def \i {\mathrm{i}}

\DeclareMathOperator{\diag}{diag}
\DeclareMathOperator{\range}{range}

\newcounter{parentalgorithm}

\makeatother

\newtheorem{theorem}{Theorem}[section]
\newtheorem{lemma}{Lemma}[section]

\newtheorem{definition}{Definition}[section]

\theoremstyle{remark}
\newtheorem{remark}{\bf Remark}[section]

\numberwithin{equation}{section}

\begin{document}

\title{\bfseries A structure-aware construction of quantum BPX preconditioners  for Q1 finite element discretizations of reaction--diffusion problems}

\author{Yin Yang\thanks{yangyinxtu@xtu.edu.cn}}
\author{Yue Yu\thanks{terenceyuyue@xtu.edu.cn}}
\author{Long Zhang}

\affil{School of Mathematics and Computational Science, Xiangtan University, Xiangtan, Hunan 411105, China}
\affil{Hunan Research Center of the Basic Discipline Fundamental Algorithmic Theory and Novel Computational Methods, Xiangtan, Hunan 411105, China}
\affil{National Center for Applied Mathematics in Hunan, Xiangtan, Hunan 411105, China}

\maketitle

\begin{abstract}
We present a structure-aware construction of quantum BPX preconditioners for Q1 finite element discretizations of constant-coefficient reaction--diffusion equations on the unit cube, including the Poisson equation. Building on existing quantum finite element and finite difference methods, we express the preconditioned operator through explicit one-dimensional maps, tensor products, and norm-preserving interlevel transfers. Symmetric preconditioning gives multilevel coefficients, and its favorable conditioning alone does not ensure efficient recovery of a state in the original finite element nodal basis. We instead use the local $L^2$-orthonormal basis underlying the construction as the output representation of the same finite element function. Its coefficient norm equals the continuous $L^2$ norm. Assuming direct access to the preconditioned right-hand side, we prepare this orthonormal coefficient state with a represented function error at most $\varepsilon$. The expected query complexity is
\[
\mathcal{O}\!\Big(dL^3\frac{\|f\|_{L^2}}{\|u\|_{L^2}}
\log\!\Big(2+\frac{L\|f\|_{L^2}}{\varepsilon}\Big)\Big),
\]
where $u$ is the exact solution, $f$ is the right-hand side, $d$ is the spatial dimension, and $L=\log_2(1/h)$ for mesh width $h$. The bound holds for $0<\varepsilon\le\|u\|_{L^2}/2$ and depends only polylogarithmically on mesh refinement for fixed continuous data. We also analyze linear-functional estimation and illustrate the construction and continuous error bounds with a one-dimensional UnitaryLab experiment.
\end{abstract}


\section{Introduction}\label{sec:introduction}

Finite element discretizations of elliptic partial differential equations lead to large sparse linear systems whose conditioning deteriorates under mesh refinement \cite{Brenner2008}. Multilevel preconditioning addresses this difficulty by combining corrections on a hierarchy of approximation spaces. The classical BPX construction \cite{BramblePasciakXu1990} and the subspace-correction framework \cite{Xu1992SubspaceCorrection} provide the foundations for such additive multilevel methods.

Quantum linear system algorithms prepare states encoding solutions without computing every coefficient, beginning with the algorithm of Harrow, Hassidim, and Lloyd \cite{HHL2009}. Subsequent developments improve the dependence on precision \cite{Childs2017QLSA} and provide a general framework for matrix transformations through block encodings \cite{Gilyen2019QSVD}. For elliptic problems, Cao et al.\ \cite{Cao2013Poisson} give a quantum algorithm and circuit construction for the discretized Poisson equation, while Montanaro and Pallister \cite{qFEM-2016} analyze finite element algorithms with discretization and output accuracy included in the comparison with classical methods. These works emphasize that input access, matrix conditioning, and the requested output must all enter the complexity analysis.

Quantum preconditioning has been studied for general linear systems \cite{Clader2013preconditioned}, including fast inversion for constructing inverse preconditioners \cite{Tong2021FastInversion}. For multilevel discretizations, a central difficulty is that multiplying separate block encodings of the system matrix and the preconditioner can retain the unfavorable normalization of the original discretization. A spectral bound for the preconditioned matrix alone is therefore insufficient.

Deiml and Peterseim \cite{DeimlPeterseim2025QFEM} address this issue for Q1 discretizations of elliptic diffusion problems by constructing the preconditioned gradient through local $L^2$-orthonormal coordinates and multilevel transfers. Their analysis allows variable diffusion tensors and establishes complexity bounds for linear functionals. Jin, Liu, Ma, and Yu \cite{JinLiuMaYu2026BPXrevised} develop a related finite difference construction using a commuting identity and Schr\"odingerization. Together, these works motivate constructing the complete preconditioned operator from its discretization structure.

The present work develops this approach for the Q1 discretization of $-\Delta u+cu=f$ on the unit cube, with constant $c\ge0$. Reaction--diffusion systems also occur as the shifted elliptic subproblems in rational approximations of fractional operators. We examine a BPX factor that combines inverse mass square roots with weights depending on the reaction coefficient. The analysis tracks the dimension, refinement level, and reaction coefficient explicitly, and connects circuit accuracy to the continuous $L^2$ error of the represented finite element function.

The choice of output representation is essential. Symmetric BPX preconditioning solves for multilevel coefficients; obtaining a state of the coefficients in the original finite element nodal basis requires a further transformation whose cost is not controlled by the preconditioned condition number alone. We therefore use the local $L^2$-orthonormal basis, already introduced to construct the block encodings, to represent the output solution. This changes the coefficient representation while preserving the finite element function. The Euclidean error of these orthonormal coefficients is exactly the continuous $L^2$ error between the represented functions, so the analysis can address continuous accuracy directly. Our solution-state complexity bound concerns this orthonormal representation; it does not include an additional conversion to a nodal coefficient state.

Building on the quantum finite element framework of \cite{DeimlPeterseim2025QFEM}, we focus on the parameters and coefficient representations that determine the quantum cost. That work already provides the underlying orthonormal gradient construction and isometric transfers, with a block-encoding normalization proportional to $dL$ for unit diffusion. The present contributions refine this framework in the following four respects.
\begin{itemize}
\item \textbf{Explicit dimension dependence in the quantum error and complexity analysis.}
We give a finite element $L^2$ error constant of $3$, an effective condition-number bound of $12dL^2$ for the BPX system, and a matrix block-encoding normalization of $12dL$. Here $d$ is the spatial dimension, $L=\log_2(1/h)$ is the number of refinement levels, and $h$ is the finest mesh width; the effective condition number uses only nonzero eigenvalues. These bounds hold uniformly in the reaction coefficient $c\ge0$ and make explicit the dimension dependence left implicit in the BPX condition estimate of \cite{DeimlPeterseim2025QFEM}.

\item \textbf{Solution-state preparation under a continuous $L^2$ error requirement.}
Let $u$ be the exact solution with right-hand side $f$, and let $u_h^q$ denote the function represented by the output coefficients in the local $L^2$-orthonormal basis. Assuming direct access to the preconditioned right-hand side, we establish $\|u-u_h^q\|_{L^2}\le\varepsilon$ for a prescribed accuracy $0<\varepsilon\le\|u\|_{L^2}/2$. For fixed $\|f\|_{L^2}/\|u\|_{L^2}$, the query bound has no algebraic inverse-mesh factor, including the cost of converting multilevel coefficients to this orthonormal coefficient state. This complements the linear-functional complexity result in \cite[Theorem~6.5]{DeimlPeterseim2025QFEM}.

\item \textbf{A structure-aware block encoding of the preconditioned system.}
Following the approach of \cite{JinLiuMaYu2026BPXrevised}, we exploit the nested Cartesian grids to express the required maps through explicit one-dimensional matrices and tensor products. Commuting identities move differentiation to each coarse level, where the BPX weights balance its scale, and norm-preserving transfers then lift the coefficients to the finest level. This gives a block encoding of the complete preconditioned operator with controlled normalization. By specializing the framework of \cite{DeimlPeterseim2025QFEM} to this tensor structure, we obtain a transparent sequence of encoding steps with explicit normalization factors.

\item \textbf{A modified BPX factor connecting quantum coefficients to continuous norms.}
The level-$\ell$ block of our BPX factor is $S_\ell=(h_\ell^{-2}+c)^{-1/2}P_\ell M_\ell^{-1/2}$, where $h_\ell$ is the mesh width, $M_\ell$ is the mass matrix, and $P_\ell$ prolongs nodal coefficients to the finest grid. The inverse mass square root replaces scalar mass scaling. On each level, the Euclidean norm of the coefficients in the mass-normalized basis equals the represented function's $L^2$ norm, giving an exact correspondence without additional mesh- or dimension-dependent norm-equivalence factors.
\end{itemize}

These query bounds exclude construction of the supplied input oracles and classical readout of the full solution. A one-dimensional UnitaryLab experiment illustrates finite element convergence, the preconditioned spectrum, and the accuracy of the represented response; it does not serve as evidence of asymptotic quantum advantage.

In Section~\ref{sec:bpx}, we formulate the Q1 discretization and BPX system and establish their error and spectral bounds. Section~\ref{sec:encoding} develops the block encodings from local coefficients and interlevel transfers. We analyze continuous accuracy and the two output tasks in Section~\ref{sec:accuracy}. Numerical results are presented in Section~\ref{sec:numerics}, and Section~\ref{sec:discussion} compares the costs and discusses the conclusions. Appendix~\ref{app:feconstant} gives the proof of an explicit finite element error constant.

\section{Multilevel BPX preconditioning for Q1 finite elements}\label{sec:bpx}

Let $\Omega=(0,1)^d$. We give a review of a BPX preconditioner for the Q1 finite element discretization of a reaction--diffusion equation with homogeneous Dirichlet boundary conditions:
\begin{equation}\label{poisson}
-\Delta u+cu=f\quad\text{in }\Omega,\qquad u=0\quad\text{on }\partial\Omega,
\end{equation}
where $c>0$ is a constant reaction coefficient and $f\in L^2(\Omega)$.

\subsection{Q1 finite element discretization}\label{sec:model}

Let $V=H_0^1(\Omega)$. The variational formulation of \eqref{poisson} is to find $u\in V$ such that
\[
a(u,v)=(f,v)\qquad\text{for all }v\in V,
\]
where
\[
a(w,v):=(\nabla w,\nabla v)+c(w,v),\qquad w,v\in V,
\]
and $(\cdot,\cdot)$ denotes the $L^2(\Omega)$ inner product.

We next consider a conforming Q1 finite element discretization. For $\ell=1,\cdots,L$, let $\mathcal T_\ell$ be the uniform Cartesian partition of $\Omega$ into cubes of side length $h_\ell=2^{-\ell}$. Define
\[
V_\ell:=\{v\in V:v|_T\in\mathbb Q_1(T)\text{ for every }T\in\mathcal T_\ell\},
\]
where $\mathbb Q_1(T)$ consists of polynomials of degree at most one in each coordinate. The spaces are nested, $V_1\subset\cdots\subset V_L$, and $\dim V_\ell=(2^\ell-1)^d$. Set $h=h_L$, $V_h=V_L$, and $N=\dim V_h=(2^L-1)^d$. The discrete variational problem is to find $u_h\in V_h$ such that
\[
a(u_h,v_h)=(f,v_h)\qquad\text{for all }v_h\in V_h.
\]

To express this problem in nodal coordinates, let $x_j^{(\ell)}=jh_\ell$, $0\le j\le2^\ell$, be the one-dimensional mesh nodes. For $1\le j\le2^\ell-1$, the nodal basis function $\phi_j^{(\ell)}$ is continuous and affine on each mesh interval, and satisfies $\phi_j^{(\ell)}(x_i^{(\ell)})=\delta_{ij}$ for $0\le i\le2^\ell$, where $\delta_{ij}$ is the Kronecker delta. For the spatial coordinate vector $\bm{x}=(x_1,\cdots,x_d)\in\Omega$, the tensor-product basis of $V_\ell$ is
\[
\phi_{\bm{j}}^{(\ell)}(\bm{x})=\prod_{q=1}^d\phi_{j_q}^{(\ell)}(x_q),
\qquad \bm{j}=(j_1,\cdots,j_d)\in \mathcal{I}_{\ell,d}: = \{1,\cdots,2^\ell-1\}^d.
\]
The finite element solution has the expansion
\[
u_h=\sum_{\bm{j}\in\mathcal{I}_{\ell,d}}
(\bm{u}_h)_{\bm{j}}\phi_{\bm{j}}^{(L)},\qquad \bm{u}_h\in\mathbb R^N.
\]
The level-$\ell$ mass and stiffness matrices are defined by
\[
(M_\ell)_{\bm{i},\bm{j}}=(\phi_{\bm{j}}^{(\ell)},\phi_{\bm{i}}^{(\ell)}),\qquad
(A_\ell)_{\bm{i},\bm{j}}=(\nabla\phi_{\bm{j}}^{(\ell)},\nabla\phi_{\bm{i}}^{(\ell)}).
\]
The one-dimensional mass and stiffness matrices, both of size $(2^\ell-1)\times(2^\ell-1)$, are
\begin{equation}\label{onedmatrices}
M_{h_\ell}=\frac{h_\ell}{6}
\begin{bmatrix}
4&1&&&0\\
1&4&1&&\\
&\ddots&\ddots&\ddots&\\
&&1&4&1\\
0&&&1&4
\end{bmatrix}, \qquad
A_{h_\ell} =\frac1{h_\ell}
\begin{bmatrix}
2&-1&&&0\\
-1&2&-1&&\\
&\ddots&\ddots&\ddots&\\
&&-1&2&-1\\
0&&&-1&2
\end{bmatrix}.
\end{equation}
The tensor-product basis and separation of the integrals yield
\begin{equation}\label{tensormatrices}
M_\ell=M_{h_\ell}^{\otimes d},\qquad
A_\ell=\sum_{q=1}^d M_{h_\ell}^{\otimes(q-1)}\otimes A_{h_\ell}\otimes M_{h_\ell}^{\otimes(d-q)}.
\end{equation}
The matrix representing the full bilinear form $a$ on $V_\ell$ is
\begin{equation}\label{reactionmatrix}
K_\ell=A_\ell+cM_\ell,\qquad
K_{h,d}=K_L=A_{h,d}+cM_{h,d},
\end{equation}
where $A_{h,d}=A_L$ and $M_{h,d}=M_L$.
The discrete variational problem is therefore equivalent to
\begin{equation}\label{femsystem}
K_{h,d}\bm{u}_h=\bm{f}_{h,d},\qquad
(\bm{f}_{h,d})_{\bm{j}}=(f,\phi_{\bm{j}}^{(L)}).
\end{equation}

For later norm estimates, define $\mathcal A_h:V_h\to V_h$ by
$(\mathcal A_hv_h,w_h)=a(v_h,w_h)$ for all $v_h,w_h\in V_h$.
If $\bm{v}$ and $\bm{w}$ are the nodal vectors of $v_h$ and $\mathcal A_hv_h$, respectively, this definition gives $M_{h,d}\bm{w}=K_{h,d}\bm{v}$. Thus the coefficient matrix of $\mathcal A_h$ is $M_{h,d}^{-1}K_{h,d}=M_{h,d}^{-1}A_{h,d}+cI$. Let $\Pi_h$ denote the $L^2$-orthogonal projection onto $V_h$. Equation~\eqref{femsystem} is equivalently $u_h=\mathcal A_h^{-1}\Pi_hf$.

For each level $\ell$, define the one-dimensional sine vectors $\bm{s}_j^{(\ell)}\in\mathbb R^{2^\ell-1}$ by
\[
(\bm{s}_j^{(\ell)})_i=\sqrt{\frac{2}{2^\ell}}\sin\!\Big(\frac{\pi ij}{2^\ell}\Big),
\qquad 1\le i,j\le2^\ell-1.
\]
They form an orthonormal basis in the Euclidean inner product and satisfy $A_{h_\ell}\bm{s}_j^{(\ell)}=\lambda_j^{(\ell)}\bm{s}_j^{(\ell)}$ and $M_{h_\ell}\bm{s}_j^{(\ell)}=\mu_j^{(\ell)}\bm{s}_j^{(\ell)}$, where
\[
\lambda_j^{(\ell)}=\frac4{h_\ell}\sin^2\frac{\pi j}{2^{\ell+1}},\qquad
\mu_j^{(\ell)}=\frac{h_\ell}{3}\Big(2+\cos\frac{\pi j}{2^\ell}\Big).
\]
Let $v_h\in V_h\setminus\{0\}$ have nodal coefficient vector $\bm{v}$. Since $\mathcal A_h$ is represented by $M_{h,d}^{-1}K_{h,d}$, the eigenvalue equation $\mathcal A_hv_h=\theta v_h$ is equivalent to
\[
M_{h,d}^{-1}K_{h,d}\bm{v}=\theta\bm{v}
\quad\Longleftrightarrow\quad
K_{h,d}\bm{v}=\theta M_{h,d}\bm{v}.
\]
The values $\theta$ satisfying the latter equation for some $\bm{v}\ne0$ are called the generalized eigenvalues of the matrix pair $(K_{h,d},M_{h,d})$. On the finest mesh, the corresponding $d$-dimensional basis vectors are
\[
\bm{s}_{\bm{j}}=\bm{s}_{j_1}^{(L)}\otimes\cdots\otimes\bm{s}_{j_d}^{(L)},
\qquad \bm{j}=(j_1,\cdots,j_d)\in \mathcal{I}_{L,d}.
\]
These $N$ vectors form an orthonormal basis of the nodal coefficient space $\mathbb R^N$. Substituting $\bm{v}=\bm{s}_{\bm{j}}$ into the generalized eigenvalue equation and using \eqref{tensormatrices} and \eqref{reactionmatrix} gives
\begin{equation}\label{generalizedspectrum}
\theta_{\bm{j}}=c+\sum_{q=1}^d\frac{\lambda_{j_q}^{(L)}}{\mu_{j_q}^{(L)}}, \qquad 
c+4d\le\theta_{\bm{j}}\le c+\frac{12d}{h^2}.
\end{equation}
These eigenvalues describe $\mathcal A_h$ in the $L^2$ inner product and will be used to compare the norms of $\Pi_hf$ and $u_h$. The reaction term shifts each diffusion eigenvalue by $c$. The upper bound in \eqref{generalizedspectrum} follows from $\lambda_j^{(L)}\le4/h$ and $\mu_j^{(L)}\ge h/3$; the lower bound follows from $\sin(\pi/2^{L+1})\ge2^{-L}$ and $\mu_j^{(L)}\le h$. The mass eigenvalues also give
\begin{equation}\label{massbounds}
\frac{h_\ell^d}{3^d}I\preceq M_\ell\preceq h_\ell^d I.
\end{equation}

\begin{lemma}[Finite element error]\label{lem:feerror}
Suppose $u\in H^2(\Omega)\cap H_0^1(\Omega)$. The Q1 Galerkin approximation satisfies
\begin{equation}\label{feerror}
\|u-u_h\|_{L^2}\le C_{\mathrm{FE}}h^2\|u\|_{H^2},
\qquad C_{\mathrm{FE}}=3.
\end{equation}
\end{lemma}
The proof in Appendix~\ref{app:feconstant} is included solely to give an explicit upper bound for $C_{\mathrm{FE}}$.

\subsection{BPX preconditioner for Q1 finite elements}

The classical BPX preconditioner combines level corrections additively \cite{BramblePasciakXu1990}; its analysis fits the subspace-correction framework of \cite{Xu1992SubspaceCorrection}. Here we use this multilevel structure with mass normalization and reaction-dependent weights.

Let $P_\ell\in\mathbb R^{N\times(2^\ell-1)^d}$ be the nodal prolongation from level $\ell$ to the finest level $L$. It represents the inclusion $V_\ell\subset V_L$: multiplying a coarse nodal vector by $P_\ell$ gives the fine nodal vector of the same function. In particular, $P_L=I$.

To balance the diffusion scale $h_\ell^{-2}$ and the reaction scale $c$, define
\begin{equation}\label{bpxfactor}
\begin{aligned}
\omega_\ell&=(h_\ell^{-2}+c)^{-1/2},\qquad
S_\ell=\omega_\ell P_\ell M_\ell^{-1/2},\\
S&=[S_1,\ \cdots,\ S_L],\qquad
B=SS^T=\sum_{\ell=1}^L\omega_\ell^2P_\ell M_\ell^{-1}P_\ell^T.
\end{aligned}
\end{equation}
The Poisson BPX constructions in \cite[Section~5.2]{DeimlPeterseim2025QFEM} and \cite[Section~3.3]{JinLiuMaYu2026BPXrevised} use, in our notation, the level factor $h_\ell^{1-d/2}P_\ell=h_\ell P_\ell(h_\ell^{-d/2}I)$. Here $h_\ell^{-d/2}$ supplies the mesh scaling associated with the $L^2$ norm. In \eqref{bpxfactor}, we use the symmetric inverse mass square root $M_\ell^{-1/2}$ in place of $h_\ell^{-d/2}I$ to normalize the level coefficients with respect to the $L^2$ inner product. In fact, for $v_\ell\in V_\ell$ with nodal coefficient vector $\bm{v}_\ell$, define the normalized coefficient vector $\bm{z}_\ell=M_\ell^{1/2}\bm{v}_\ell$. Then
\[
\|v_\ell\|_{L^2}^2=\bm{v}_\ell^TM_\ell\bm{v}_\ell=\|\bm{z}_\ell\|_2^2.
\]
Thus continuous $L^2$ errors on level $\ell$ correspond to Euclidean errors in the normalized coefficients $\bm{z}_\ell$, whereas errors in the nodal coefficients $\bm{v}_\ell$ are measured in the mass-matrix norm.

The weight $\omega_\ell$ behaves like $h_\ell$ when $ch_\ell^2\ll1$, and like $c^{-1/2}$ when $ch_\ell^2\gg1$. Thus the same factor controls both terms of the energy. When $c=0$, it reduces to $S_\ell=h_\ell P_\ell M_\ell^{-1/2}$.
The preconditioner $B$ is related to the BPX preconditioners in \cite{DeimlPeterseim2025QFEM,JinLiuMaYu2026BPXrevised}, extended here with the reaction-dependent weights $\omega_\ell$, as follows:
\[
\sum_{\ell=1}^L\omega_\ell^2h_\ell^{-d}P_\ell P_\ell^T
\preceq B\preceq
3^d\sum_{\ell=1}^L\omega_\ell^2h_\ell^{-d}P_\ell P_\ell^T.
\]

The prolongation matrices admit an explicit tensor-product construction. We first describe one refinement step in one dimension.

Let $v$ be a continuous piecewise linear function on the one-dimensional level-$\ell$ mesh, with zero boundary values. Let $\bm{v}=(v_1,\cdots,v_{2^\ell-1})^T$ contain its interior nodal values, where $v_j=v(x_j^{(\ell)})$ and $v_0=v_{2^\ell}=0$. Refinement inserts the midpoint of every coarse interval, so $h_{\ell+1}=h_\ell/2$. The prolongation matrix
\[
P_{\ell\to\ell+1}^{\mathrm{1d}}
\in\mathbb R^{(2^{\ell+1}-1)\times(2^\ell-1)}
\]
converts $\bm{v}$ into the vector $\bm{w}$ of values of the same function at the fine interior nodes: $\bm{w}=P_{\ell\to\ell+1}^{\mathrm{1d}}\bm{v}$, with $w_i=v(x_i^{(\ell+1)})$. Its rows correspond to fine interior nodes and its columns to coarse interior nodes, both ordered from left to right.

An even fine index corresponds to an existing coarse node, since $x_{2j}^{(\ell+1)}=x_j^{(\ell)}$. An odd fine index corresponds to the midpoint of a coarse interval, where linearity gives the average of its endpoint values. Thus
\[
\begin{aligned}
w_{2j}&=v_j, &&1\le j\le2^\ell-1,\\
w_{2j+1}&=\frac{v_j+v_{j+1}}2, &&0\le j\le2^\ell-1.
\end{aligned}
\]
For example, refinement from level $2$ to level $3$ maps three interior nodal values to seven:
\[
\begin{bmatrix}w_1\\w_2\\w_3\\w_4\\w_5\\w_6\\w_7\end{bmatrix}
=P_{2\to3}^{\mathrm{1d}}\begin{bmatrix}v_1\\v_2\\v_3\end{bmatrix}
=\begin{bmatrix}
\tfrac12&0&0\\
1&0&0\\
\tfrac12&\tfrac12&0\\
0&1&0\\
0&\tfrac12&\tfrac12\\
0&0&1\\
0&0&\tfrac12
\end{bmatrix}
\begin{bmatrix}v_1\\v_2\\v_3\end{bmatrix}.
\]
The first and last rows use the zero boundary values. Prolongation changes the nodal representation while preserving the function exactly: $\sum_{j=1}^{2^\ell-1}v_j\phi_j^{(\ell)}(x)
=\sum_{i=1}^{2^{\ell+1}-1}w_i\phi_i^{(\ell+1)}(x)$.
Composing these successive refinements gives
\[P_\ell^{\mathrm{1d}}=P_{L-1\to L}^{\mathrm{1d}}\cdots P_{\ell\to\ell+1}^{\mathrm{1d}}, \qquad P_\ell=(P_\ell^{\mathrm{1d}})^{\otimes d}.\]
Then $P_\ell$ represents the inclusion $V_\ell\subset V_L$, so
\[
P_\ell^TA_{h,d}P_\ell=A_\ell,\qquad P_\ell^TM_{h,d}P_\ell=M_\ell,\qquad P_\ell^TK_{h,d}P_\ell=K_\ell.
\]

The classical left-preconditioned system is $BK_{h,d}\bm{u}_h=B\bm{f}_{h,d}$. For a symmetric quantum encoding, introduce multilevel coordinates $\bm{z}=(\bm{z}_1^T,\cdots,\bm{z}_L^T)^T$ and solve
\begin{equation}\label{bpxsystem}
A_S\bm{z}=S^T\bm{f}_{h,d},\qquad
A_S=S^TK_{h,d}S,\qquad \bm{u}_h=S\bm{z}.
\end{equation}
In particular, each level equation is
\begin{equation}\label{bpxlevelsystem}
\sum_{\ell'=1}^L \omega_\ell\omega_{\ell'}M_\ell^{-1/2}P_\ell^TK_{h,d}P_{\ell'}M_{\ell'}^{-1/2}\bm{z}_{\ell'}
=\omega_\ell M_\ell^{-1/2}P_\ell^T\bm{f}_{h,d}.
\end{equation}
The off-diagonal blocks in \eqref{bpxlevelsystem} couple corrections on different levels. The matrix $A_S$ has order $\sum_{\ell=1}^L(2^\ell-1)^d$. We embed its coordinate space in $\lceil\log_2L\rceil+dL$ qubits: the level is stored as $\ket{\ell-1}$, and each spatial register is padded to $L$ qubits. Inactive level labels and spatial indices correspond to zero matrix blocks.

Since $S_L=\omega_LM_L^{-1/2}$ is invertible, $S$ has full row rank. However, its columns are redundant when $L>1$, and $\ker A_S=\ker S$. The right-hand side of \eqref{bpxsystem} belongs to $\range(A_S)=\range(S^T)$, so the system is compatible. Its minimum-norm solution is
\begin{equation}\label{bpxinverseidentity}
\bm{z}=A_S^+S^T\bm{f}_{h,d},\qquad
SA_S^+S^T=K_{h,d}^{-1},
\end{equation}
where $A_S^+$ is the Moore--Penrose pseudoinverse. All solutions of the redundant system recover the same nodal solution, because differences between them lie in $\ker S$.

The following spectral bounds follow from an $L^2$ projection decomposition and the finite element inverse estimate. Their explicit dependence on $d$ and $L$ will determine the inverse-query bound.

\begin{theorem}[Spectral bounds for the preconditioned matrix]\label{thm:bpxspectrum}
For every $c\ge0$, every nonzero eigenvalue of $A_S$ satisfies
\begin{equation}\label{bpxspectralbound}
\frac1L\le\lambda(A_S)\le12dL.
\end{equation}
The effective condition number, defined on $\range(A_S)$, is consequently at most $12dL^2$, uniformly in $c$. These nonzero eigenvalues are also the eigenvalues of $B^{1/2}K_{h,d}B^{1/2}$.
\end{theorem}
\begin{proof}
Let $Q_\ell$ be the $L^2$-orthogonal projection onto $V_\ell$ and set $Q_0=0$. The one-dimensional approximation estimate in \cite[Theorem~4]{SandeManniSpeleers2020}, together with the best-approximation property and tensor-product orthogonality of the $L^2$ projection, yields
\begin{equation}\label{projectionestimate}
\|v-Q_\ell v\|_{L^2}\le\frac{h_\ell}{\pi}\|\nabla v\|_{L^2},\qquad v\in H_0^1(\Omega).
\end{equation}
For $v\in V_L$, put $v_\ell=(Q_\ell-Q_{\ell-1})v$. Then $v=\sum_{\ell=1}^Lv_\ell$ and orthogonality gives $\sum_{\ell=1}^L\|v_\ell\|_{L^2}^2=\|v\|_{L^2}^2$. For $\ell\ge2$, nested orthogonal projections and \eqref{projectionestimate} give
\[
h_\ell^{-2}\|v_\ell\|_{L^2}^2\le\frac4{\pi^2}\|\nabla v\|_{L^2}^2.
\]
For $\ell=1$, the Poincar\'e inequality on the cube gives
$h_1^{-2}\|Q_1v\|_{L^2}^2\le4\|\nabla v\|_{L^2}^2/(d\pi^2)$. Hence
\begin{equation}\label{stabledecomposition}
\begin{aligned}
\sum_{\ell=1}^L\omega_\ell^{-2}\|v_\ell\|_{L^2}^2
&=\sum_{\ell=1}^L h_\ell^{-2}\|v_\ell\|_{L^2}^2+c\|v\|_{L^2}^2\\
&\le\frac{4L}{\pi^2}\|\nabla v\|_{L^2}^2+c\|v\|_{L^2}^2 
 \le L\bigl(\|\nabla v\|_{L^2}^2+c\|v\|_{L^2}^2\bigr).
\end{aligned}
\end{equation}
If $\bm{v}_\ell$ is the level-$\ell$ nodal vector of $v_\ell$, choose $\bm{z}_\ell=\omega_\ell^{-1}M_\ell^{1/2}\bm{v}_\ell$. Then $S\bm{z}=\bm{v}$, where $\bm{v}$ is the fine nodal vector of $v$, and $\|\bm{z}\|_2^2$ equals the sum in \eqref{stabledecomposition}. The minimum-norm identity
\[
\bm{v}^TB^{-1}\bm{v}=\min_{S\bm{z}=\bm{v}}\|\bm{z}\|_2^2
\]
therefore implies $B^{-1}\preceq LK_{h,d}$ and proves the lower eigenvalue bound.

For the upper bound, the diffusion part of \eqref{generalizedspectrum} at level $\ell$ gives
$\|\nabla v_\ell\|_{L^2}^2\le12d h_\ell^{-2}\|v_\ell\|_{L^2}^2$. Consequently,
\[
\|K_{h,d}^{1/2}S_\ell\|^2
=\omega_\ell^2\|M_\ell^{-1/2}K_\ell M_\ell^{-1/2}\|
\le\frac{12d\,h_\ell^{-2}+c}{h_\ell^{-2}+c}\le12d.
\]
Cauchy--Schwarz over the $L$ level blocks yields $\|K_{h,d}^{1/2}S\|^2\le12dL$. Finally, the nonzero spectra of $S^TK_{h,d}S$ and $K_{h,d}^{1/2}BK_{h,d}^{1/2}$ coincide. The latter is similar to $BK_{h,d}$, as is $B^{1/2}K_{h,d}B^{1/2}$, which completes the proof.
\end{proof}

The estimate $12dL^2$ is sufficient for the analysis below, although its dependence on the number of levels is not asserted to be optimal. It is uniform in $c$ and depends on the mesh through $L=\log_2(1/h)$.

\section{Block encoding through local value and gradient coordinates}\label{sec:encoding}

We construct the block encoding through Gram factorizations of the mass and stiffness matrices, following the strategies in \cite{DeimlPeterseim2025QFEM,JinLiuMaYu2026BPXrevised}. Local $L^2$-orthonormal coordinates convert function inner products into Euclidean coefficient products and make interlevel transfer isometric. We first construct these maps, then combine them with the inverse mass square roots and BPX weights. 

\subsection{Block-encoding preliminaries}\label{subsec:encodingconvention}

Following \cite[Section~4.1]{JinLiuMaYu2026BPXrevised}, we first recall block encodings of square matrices and then extend the definition to the rectangular maps used below. All matrix errors in this subsection are measured in the operator norm.

\begin{definition}[Block encoding]\label{def:blockencoding}
Let $D\in\mathbb C^{N\times N}$, and let $\widehat D\in\mathbb C^{2^n\times2^n}$ be its zero-padded extension, where $n=\lceil\log_2N\rceil$. A unitary $U_D$ acting on $a+n$ qubits is an $(\alpha_D,a,\epsilon_D)$-block encoding of $D$ if $\alpha_D\ge\|D\|$, $\alpha_D>0$, and
\[
\Big\|\widehat D-\alpha_D
(\bra{0}^{\otimes a}\otimes I_{2^n})U_D
(\ket{0}^{\otimes a}\otimes I_{2^n})\Big\|\le\epsilon_D.
\]
The encoding is exact when $\epsilon_D=0$.
\end{definition}

The value, gradient, and interlevel transfer matrices generally have different input and output dimensions. We therefore use generalized block encodings, which allow these two spaces to be embedded separately in the circuit register \cite{Gilyen2019QSVD,DeimlPeterseim2025QFEM}.

\begin{definition}[Generalized block encoding]\label{def:generalizedencoding}
Let $D\in\mathbb C^{N_{\mathrm{out}}\times N_{\mathrm{in}}}$, and let
\[
\Pi_{\mathrm{in}}:\mathbb C^{N_{\mathrm{in}}}\longrightarrow\mathbb C^{2^m},
\qquad
\Pi_{\mathrm{out}}:\mathbb C^{N_{\mathrm{out}}}\longrightarrow\mathbb C^{2^m}
\]
be isometries, so that $\Pi_{\mathrm{in}}^\dagger\Pi_{\mathrm{in}}=I_{N_{\mathrm{in}}}$ and $\Pi_{\mathrm{out}}^\dagger\Pi_{\mathrm{out}}=I_{N_{\mathrm{out}}}$. A unitary $U_D$ on $m$ qubits is a generalized $(\alpha_D,\epsilon_D)$-block encoding of $D$ if $\alpha_D\ge\|D\|$, $\alpha_D>0$, and
\[
\Big\|D-\alpha_D\Pi_{\mathrm{out}}^\dagger U_D\Pi_{\mathrm{in}}\Big\|\le\epsilon_D.
\]
For an exact encoding, this reduces to
\[
\Pi_{\mathrm{out}}^\dagger U_D\Pi_{\mathrm{in}}=D/\alpha_D.
\]
\end{definition}

Here the isometries initialize auxiliary registers and embed the active indices into qubit registers; the corresponding orthogonal projectors are $\Pi_{\mathrm{in}}\Pi_{\mathrm{in}}^\dagger$ and $\Pi_{\mathrm{out}}\Pi_{\mathrm{out}}^\dagger$. For an exact encoding and a unit input vector $\bm{x}$, projecting $U_D\Pi_{\mathrm{in}}\bm{x}$ onto the output subspace gives $\Pi_{\mathrm{out}}D\bm{x}/\alpha_D$, with success probability $\|D\bm{x}\|_2^2/\alpha_D^2$. Thus the normalization affects the cost of using the encoded matrix. We describe the constructions in the ideal arithmetic model and account for finite precision in Section~\ref{sec:accuracy}.

The following rules are standard consequences of generalized block-encoding arithmetic \cite{Gilyen2019QSVD}.
\begin{lemma}[Matrix arithmetic for generalized block-encodings]\label{lem:arithmeticBE}
Let $U_1$ be a generalized $(\gamma_1,\varepsilon_1)$-block-encoding of $S_1$ with respect to $(\Pi_0,\Pi_1)$, and let $U_2$ be a generalized $(\gamma_2,\varepsilon_2)$-block-encoding of $S_2$ with respect to $(\Pi_1,\Pi_2)$. Assume that the embeddings can be implemented. Then:
\begin{enumerate}
\item \textbf{Adjoint.} $U_1^\dagger$ is a generalized $(\gamma_1,\varepsilon_1)$-block-encoding of $S_1^\dagger$ with respect to $(\Pi_1,\Pi_0)$.
\item \textbf{Product.} Composing the encodings with separate flags and the intermediate projection gives a generalized $(\gamma_1\gamma_2,\gamma_2\varepsilon_1+\gamma_1\varepsilon_2)$-block-encoding of $S_2S_1$ with respect to $(\Pi_0,\Pi_2)$.
\item \textbf{Tensor product.} Let $U'_1$ be a generalized $(\gamma'_1,\varepsilon'_1)$-block-encoding of $S'_1$ with respect to $(\Pi'_0,\Pi'_1)$, acting on a disjoint register. Then $U_1\otimes U'_1$ is a generalized $(\gamma_1\gamma'_1,\gamma'_1\varepsilon_1+\gamma_1\varepsilon'_1)$-block-encoding of $S_1\otimes S'_1$ with respect to $(\Pi_0\otimes\Pi'_0,\Pi_1\otimes\Pi'_1)$.
\item \textbf{Direct sum.} If $U_j$ encodes $S_j$ with parameters $(\gamma,\varepsilon_j)$ for every $j$, then $\sum_j\ket{j}\bra{j}\otimes U_j$ is a generalized $(\gamma,\max_j\varepsilon_j)$-block-encoding of $\bigoplus_jS_j$ with the corresponding direct-sum embeddings.
\item \textbf{Concatenation.} Let $U_A$ and $U_B$ be generalized $(\gamma_A,\varepsilon_A)$- and $(\gamma_B,\varepsilon_B)$-block-encodings of $A$ and $B$. With a common output embedding, $[A\ B]$ admits a generalized $(\sqrt{\gamma_A^2+\gamma_B^2},\sqrt{\varepsilon_A^2+\varepsilon_B^2})$-block-encoding. With a common input embedding, the same parameters apply to $[A^\dagger\ B^\dagger]^\dagger$. For several blocks, the parameters are the square roots of the corresponding sums of squares.
\end{enumerate}
\end{lemma}

Applying the product rule directly to separate encodings of $S^T$, $K_{h,d}$, and $S$ produces an encoding of $A_S=S^TK_{h,d}S$ with normalization
\begin{equation}\label{directproductnormalization}
\alpha_{A_S}^{\mathrm{prod}}
=\alpha_S^2\alpha_{K_{h,d}}
\ge\|S\|^2\|K_{h,d}\|
=\|B\|\,\|K_{h,d}\|.
\end{equation}
The product normalization depends on the separate factors and can be much larger than the norm of the preconditioned matrix. Indeed, Theorem~\ref{thm:bpxspectrum} gives $B\succeq L^{-1}K_{h,d}^{-1}$, and hence
\[
\alpha_{A_S}^{\mathrm{prod}}\ge\frac{\kappa_2(K_{h,d})}{L},\qquad
\kappa_2(K_{h,d})=\|K_{h,d}\|\,\|K_{h,d}^{-1}\|.
\]
For fixed $d$ and $c$, the sine eigenvalues imply $\kappa_2(K_{h,d})=\Theta(h^{-2})$ under refinement. Thus even optimally normalized separate factors retain algebraic mesh dependence when combined by the product rule. Since the inverse-query cost depends on the encoding normalization relative to the smallest nonzero eigenvalue, the spectral improvement alone is insufficient.

Following \cite[Section~4.2]{JinLiuMaYu2026BPXrevised}, we instead construct the complete level contributions to the energy map $\mathcal C$ in \eqref{compositemaps}. The commuting identities move differentiation to level $\ell$, where its scale is balanced by $\omega_\ell$, before transferring the coefficients to the finest mesh. The Gram product $A_S=\mathcal C^T\mathcal C$ then has normalization $12dL$, as shown in Theorem~\ref{thm:encoding}.

The construction below proceeds from local value and gradient matrices to interlevel transfers, normalized level encodings, and finally the BPX system.

\subsection{The value and gradient matrices}

Following the Gram-factorization strategy used for the finite difference operator in \cite[Section~4.2]{JinLiuMaYu2026BPXrevised}, we seek factors of the finite element mass and stiffness matrices. We begin in one dimension and aim to construct the matrices in Lemma~\ref{lem:bpx1dfactors} so that
\[
M_{h_\ell}=(E_\ell^{\mathrm{1d}})^TE_\ell^{\mathrm{1d}},
\qquad
A_{h_\ell}=(C_\ell^{\mathrm{1d}})^TC_\ell^{\mathrm{1d}}.
\]
For arbitrary nodal vectors $\bm{v},\bm{w}\in\mathbb R^{2^\ell-1}$, let $v_h=\sum_jv_j\phi_j^{(\ell)}$ and $w_h=\sum_jw_j\phi_j^{(\ell)}$. The desired factorizations are equivalent to the identities
\[
\begin{aligned}
(E_\ell^{\mathrm{1d}}\bm{v})^T(E_\ell^{\mathrm{1d}}\bm{w})
 = \bm{v}^T(E_\ell^{\mathrm{1d}})^TE_\ell^{\mathrm{1d}}\bm{w}
 =\bm{v}^TM_{h_\ell}\bm{w}
 =(v_h,w_h)_{L^2(0,1)},\\
(C_\ell^{\mathrm{1d}}\bm{v})^T(C_\ell^{\mathrm{1d}}\bm{w})
 = \bm{v}^T(C_\ell^{\mathrm{1d}})^TC_\ell^{\mathrm{1d}}\bm{w}
 =\bm{v}^TA_{h_\ell}\bm{w}
 = (v_h',w_h')_{L^2(0,1)}.
\end{aligned}
\]
These identities require the coefficient inner products to reproduce the $L^2$ inner products of the functions and their derivatives. An $L^2$-orthonormal basis provides exactly this property. Cellwise constant and linear modes represent both the finite element functions and their derivatives, so they give the required one-dimensional factors. We assemble the local maps and then extend them by tensor products. The resulting coordinates agree with the local representation used in \cite[Section~6.1]{DeimlPeterseim2025QFEM}.

\noindent\textbf{One-dimensional local coefficients.}
Fix a level $\ell$ and let $x_i=ih_\ell$, $i=0,\cdots,2^\ell$. A piecewise linear finite element function $v_h\in V_\ell\subset V_h$ is determined by its interior nodal vector
$\bm{v}=(v_1,\cdots,v_{2^\ell-1})^T$, where $v_j=v_h(x_j)$ and $v_0=v_{2^\ell}=0$ at the boundary. Equivalently, $v_h=\sum_{j=1}^{2^\ell-1}v_j\phi_{j}^{(\ell)}$, where $\phi_{j}^{(\ell)}$ is the piecewise linear nodal basis function associated with $x_j$. On the interval $I_{i}^{(\ell)}=(x_i,x_{i+1})$, write $\varphi=(x-x_i)/h_\ell$. Then
\[
v_h(x)=(1-\varphi)v_i+\varphi v_{i+1},
\qquad v_h'(x)=\frac{v_{i+1}-v_i}{h_\ell}.
\]
To realize the required inner-product identities, choose the following $L^2$-orthonormal basis on each interval:
\begin{equation}\label{bpxlocalbasis}
\psi_{i,0}^{(\ell)}(x)=\frac1{\sqrt{h_\ell}},\qquad
\psi_{i,1}^{(\ell)}(x)=\sqrt{\frac3{h_\ell}}(2\varphi-1),
\qquad x\in I_{i}^{(\ell)},
\end{equation}
extended by zero outside $I_{i}^{(\ell)}$. They satisfy
$\int_{I_{i}^{(\ell)}}\psi_{i,\nu}^{(\ell)}\psi_{i,\mu}^{(\ell)}\,\d x=\delta_{\nu\mu}$.
The identity
$v_h=(v_i+v_{i+1})/2+(v_{i+1}-v_i)(2\varphi-1)/2$ therefore gives the exact local expansions
\[
\begin{aligned}
v_h|_{I_{i}^{(\ell)}}
&=\frac{\sqrt{h_\ell}}2(v_i+v_{i+1})\psi_{i,0}^{(\ell)}
+\sqrt{\frac{h_\ell}{12}}(v_{i+1}-v_i)\psi_{i,1}^{(\ell)},\\
v_h'|_{I_{i}^{(\ell)}}
&=\frac{v_{i+1}-v_i}{\sqrt{h_\ell}}\psi_{i,0}^{(\ell)}.
\end{aligned}
\]
Thus the function has two local coefficients, while its derivative has only a constant coefficient. Define $E_\ell^{\mathrm{1d}}$ and $C_\ell^{\mathrm{1d}}$ to collect these coefficients over all intervals. Their action on interval $i$ is
\begin{equation}\label{bpxlocalmaps}
\begin{aligned}
\begin{bmatrix}(E_\ell^{\mathrm{1d}}\bm{v})_{i,0}\\(E_\ell^{\mathrm{1d}}\bm{v})_{i,1}\end{bmatrix}
&=\begin{bmatrix}
\sqrt{h_\ell}/2&\sqrt{h_\ell}/2\\
-\sqrt{h_\ell/12}&\sqrt{h_\ell/12}
\end{bmatrix}
\begin{bmatrix}v_i\\v_{i+1}\end{bmatrix},\\
\begin{bmatrix}(C_\ell^{\mathrm{1d}}\bm{v})_{i,0}\\(C_\ell^{\mathrm{1d}}\bm{v})_{i,1}\end{bmatrix}
&=\begin{bmatrix}
-1/\sqrt{h_\ell}&1/\sqrt{h_\ell}\\0&0
\end{bmatrix}
\begin{bmatrix}v_i\\v_{i+1}\end{bmatrix}.
\end{aligned}
\end{equation}
The two output rows correspond to the constant and linear modes on the interval. In particular, the zero row of the derivative map records that $v_h'$ is constant there.

\noindent\textbf{Assembly on the one-dimensional mesh.}
For example, stacking the local value maps in \eqref{bpxlocalmaps} for $\ell=2$, with $v_0=v_4=0$, gives
\[
\begin{bmatrix}
(E_2^{\mathrm{1d}}\bm{v})_{0,0}\\
(E_2^{\mathrm{1d}}\bm{v})_{0,1}\\
(E_2^{\mathrm{1d}}\bm{v})_{1,0}\\
(E_2^{\mathrm{1d}}\bm{v})_{1,1}\\
(E_2^{\mathrm{1d}}\bm{v})_{2,0}\\
(E_2^{\mathrm{1d}}\bm{v})_{2,1}\\
(E_2^{\mathrm{1d}}\bm{v})_{3,0}\\
(E_2^{\mathrm{1d}}\bm{v})_{3,1}
\end{bmatrix}
=\sqrt{h_2}
\left[
\begin{array}{ccc}
\tfrac12&0&0\\
\tfrac1{\sqrt{12}}&0&0\\ \hline
\tfrac12&\tfrac12&0\\
-\tfrac1{\sqrt{12}}&\tfrac1{\sqrt{12}}&0\\ \hline
0&\tfrac12&\tfrac12\\
0&-\tfrac1{\sqrt{12}}&\tfrac1{\sqrt{12}}\\ \hline
0&0&\tfrac12\\
0&0&-\tfrac1{\sqrt{12}}
\end{array}
\right]
\begin{bmatrix}v_1\\v_2\\v_3\end{bmatrix}.
\]
Each output row is indexed by $(i,\nu)$: $i$ identifies the interval $I_i^{(\ell)}$, and $\nu=0,1$ identifies the constant or linear mode. The entry of $E_\ell^{\mathrm{1d}}\bm{v}$ is the coefficient of $\psi_{i,\nu}^{(\ell)}$ in $v_h$; the corresponding entry of $C_\ell^{\mathrm{1d}}\bm{v}$ is its coefficient in $v_h'$. Within each interval the constant mode precedes the linear mode, giving the row order $(0,0),(0,1),\cdots,(2^\ell-1,0),(2^\ell-1,1)$. The columns correspond to the interior nodal values. Both matrices have size $2^{\ell+1}\times(2^\ell-1)$, with entries
\begin{equation}\label{bpxmatrixentries}
\begin{aligned}
(E_\ell^{\mathrm{1d}})_{(i,0),j}&=\frac{\sqrt{h_\ell}}2(\delta_{i,j}+\delta_{i+1,j}),\\
(E_\ell^{\mathrm{1d}})_{(i,1),j}&=\sqrt{\frac{h_\ell}{12}}(\delta_{i+1,j}-\delta_{i,j}),\\
(C_\ell^{\mathrm{1d}})_{(i,0),j}&=\frac{\delta_{i+1,j}-\delta_{i,j}}{\sqrt{h_\ell}},
\qquad (C_\ell^{\mathrm{1d}})_{(i,1),j}=0.
\end{aligned}
\end{equation}
Here $\delta_{ij}$ is the Kronecker delta. Since boundary values are fixed to zero, no boundary columns are included.

\begin{lemma}[One-dimensional mass and stiffness factorizations]\label{lem:bpx1dfactors}
For every level $\ell\ge1$, the matrices defined in \eqref{bpxmatrixentries} satisfy
\[
(E_\ell^{\mathrm{1d}})^TE_\ell^{\mathrm{1d}}=M_{h_\ell},\qquad
(C_\ell^{\mathrm{1d}})^TC_\ell^{\mathrm{1d}}=A_{h_\ell}.
\]
In particular, for $v_h=\sum_{j=1}^{2^\ell-1}v_j\phi_j^{(\ell)}$ with nodal coefficient vector $\bm{v}$,
\[
\|E_\ell^{\mathrm{1d}}\bm{v}\|_2^2=\|v_h\|_{L^2(0,1)}^2,
\qquad
\|C_\ell^{\mathrm{1d}}\bm{v}\|_2^2=\|v_h'\|_{L^2(0,1)}^2.
\]
\end{lemma}
\begin{proof}
The expansions in \eqref{bpxlocalmaps} show that $E_\ell^{\mathrm{1d}}\bm{v}$ and $C_\ell^{\mathrm{1d}}\bm{v}$ are the exact coefficients of $v_h$ and $v_h'$ in $\{\psi_{i,\nu}^{(\ell)}\}_{i,\nu}$. This family is orthonormal in $L^2(0,1)$, since it is orthonormal on each interval and has disjoint supports across intervals. Hence the coefficient inner products equal the function and derivative inner products, as required above. Since $\bm{v}$ and $\bm{w}$ are arbitrary, both Gram factorizations follow. Taking $\bm{w}=\bm{v}$ gives the norm identities.
\end{proof}

\noindent\textbf{Tensor products in $d$ dimensions.}
On the Cartesian grid, the nodal basis and the cellwise orthonormal basis are products of their one-dimensional counterparts:
\[
\phi_{\bm{j}}^{(\ell)}(\bm{x})=\prod_{q=1}^d\phi_{j_q}^{(\ell)}(x_q),
\qquad
\psi_{\bm{m},\bm{\nu}}^{(\ell)}(\bm{x})
=\prod_{q=1}^d\psi_{m_q,\nu_q}^{(\ell)}(x_q).
\]
Here $\bm{j}\in \mathcal{I}_{\ell,d}$, $\bm{m}\in \mathcal{I}_{\ell,d}$, and $\bm{\nu}\in\{0,1\}^d$. The function $\psi_{\bm{m},\bm{\nu}}^{(\ell)}$ is supported on $\prod_{q=1}^d I_{m_q}^{(\ell)}$. These functions form an orthonormal basis of the space $\mathcal Q_\ell$ of cellwise multilinear functions, with no continuity imposed across cell boundaries. There are $2^d$ modes on each of the $2^{\ell d}$ cells. We order the output indices by the coordinate pairs $(m_1,\nu_1),\cdots,(m_d,\nu_d)$.

To represent the function, apply the one-dimensional value map in every coordinate. To represent its $q$th partial derivative, apply the derivative map in coordinate $q$ and the value map in every other coordinate. For example, in two dimensions this gives
\[
E_\ell=E_\ell^{\mathrm{1d}}\otimes E_\ell^{\mathrm{1d}},\qquad
C_\ell=\begin{bmatrix}
C_\ell^{\mathrm{1d}}\otimes E_\ell^{\mathrm{1d}}\\
E_\ell^{\mathrm{1d}}\otimes C_\ell^{\mathrm{1d}}
\end{bmatrix}.
\]
In general, for $1\le q\le d$, define
\begin{equation}\label{bpxmaptensorproducts}
\begin{aligned}
E_\ell&=(E_\ell^{\mathrm{1d}})^{\otimes d},
\qquad
C_\ell=\begin{bmatrix}(C_\ell^{(1)})^T&\cdots&(C_\ell^{(d)})^T\end{bmatrix}^{T},\\[4pt]
C_\ell^{(q)}&=(E_\ell^{\mathrm{1d}})^{\otimes(q-1)}
\otimes C_\ell^{\mathrm{1d}}
\otimes(E_\ell^{\mathrm{1d}})^{\otimes(d-q)}.
\end{aligned}
\end{equation}
Thus $E_\ell\in\mathbb R^{2^{d(\ell+1)}\times(2^\ell-1)^d}$ and
$C_\ell\in\mathbb R^{d\,2^{d(\ell+1)}\times(2^\ell-1)^d}$.
For $v_h=\sum_{\bm{j}}v_{\bm{j}}\phi_{\bm{j}}^{(\ell)}$, these matrices provide the exact expansions
\begin{equation}\label{bpxmapaction}
v_h =\sum_{\bm{m},\bm{\nu}}(E_\ell\bm{v})_{\bm{m},\bm{\nu}}
\psi_{\bm{m},\bm{\nu}}^{(\ell)}, \qquad
\partial_qv_h =\sum_{\bm{m},\bm{\nu}}(C_\ell^{(q)}\bm{v})_{\bm{m},\bm{\nu}}
\psi_{\bm{m},\bm{\nu}}^{(\ell)}.
\end{equation}
Equivalently, the corresponding matrix entries are
$(\phi_{\bm{j}}^{(\ell)},\psi_{\bm{m},\bm{\nu}}^{(\ell)})$ and
$(\partial_q\phi_{\bm{j}}^{(\ell)},\psi_{\bm{m},\bm{\nu}}^{(\ell)})$. The weak derivative is computed cellwise; each $\partial_qv_h$ belongs to $\mathcal Q_\ell$. Orthonormality gives
$\|E_\ell\bm{v}\|_2^2=\|v_h\|_{L^2}^2$ and
$\|C_\ell\bm{v}\|_2^2=\sum_{q=1}^d\|\partial_qv_h\|_{L^2}^2$.
Taking the tensor products of the one-dimensional mass identities in Lemma~\ref{lem:bpx1dfactors} and summing the derivative contributions therefore yields
\begin{equation}\label{bpxenergyfactors}
E_\ell^T E_\ell=M_\ell,\qquad C_\ell^T C_\ell=A_\ell,\qquad K_\ell=C_\ell^TC_\ell+cE_\ell^TE_\ell.
\end{equation}

\subsection{Interlevel transfers and commuting identities}

Every function in $\mathcal Q_\ell$ also belongs to $\mathcal Q_L$, because a multilinear polynomial on a coarse cell remains multilinear when restricted to each of its fine subcells. For a function $w\in\mathcal Q_\ell$, let $\bm{w}_\ell$ and $\bm{w}_L$ denote its coefficient vectors in the orthonormal bases of $\mathcal Q_\ell$ and $\mathcal Q_L$, respectively. We define $T_{\ell\to L}^{(0)}$ as the matrix that converts the coefficients in the coarse-mesh basis into the coefficients of the same function in the fine-mesh basis:
\[
\bm{w}_L=T_{\ell\to L}^{(0)}\bm{w}_\ell.
\]
Both coefficient vectors represent exactly the same function $w$. To represent the gradient of a finite element function on the fine mesh, we apply this coefficient conversion separately to each of its $d$ components. The matrix that performs these $d$ conversions is denoted by $T_{\ell\to L}$. If $\beta$ labels a basis function on the coarse mesh and $\alpha$ labels a basis function on the fine mesh, the two matrices are given by
\[
(T_{\ell\to L}^{(0)})_{\alpha\beta}
=(\psi_{\beta}^{(\ell)},\psi_{\alpha}^{(L)}),
\qquad
T_{\ell\to L}=\diag(\underbrace{T_{\ell\to L}^{(0)},\cdots,T_{\ell\to L}^{(0)}}_{d\text{ blocks}}).
\]
Inclusion, differentiation, and coefficient transfer satisfy a commuting relation. It is the finite element relation underlying \cite[Section~6.1]{DeimlPeterseim2025QFEM}; a corresponding discrete-gradient identity organizes the finite difference construction in \cite{JinLiuMaYu2026BPXrevised}.
\begin{lemma}[Commuting identities for the value and gradient maps]\label{lem:bpxcommuting}
For $1\le \ell\le L$, the transfer maps $T_{\ell\to L}^{(0)}$ and $T_{\ell\to L}$ are isometries, and
\begin{equation}\label{bpxcommuting}
E_LP_\ell=T_{\ell\to L}^{(0)}E_\ell,\qquad
C_LP_\ell=T_{\ell\to L}C_\ell.
\end{equation}
\end{lemma}
\begin{proof}
Take an arbitrary coarse nodal vector $\bm{v}$ and let $v_h=\sum_{\bm{j}}v_{\bm{j}}\phi_{\bm{j}}^{(\ell)}$. Since $V_\ell\subset V_L$ and $P_\ell$ is the nodal prolongation, $P_\ell\bm{v}$ represents the same function on the fine mesh:
\[
v_h=\sum_{\bm{i}}(P_\ell\bm{v})_{\bm{i}}\phi_{\bm{i}}^{(L)}.
\]
For each fine basis index $\alpha$, the definition of $T_{\ell\to L}^{(0)}$ and the expansion \eqref{bpxmapaction} give
\[
(E_LP_\ell\bm{v})_\alpha
=(v_h,\psi_{\alpha}^{(L)})
=\sum_\beta(E_\ell\bm{v})_\beta
(\psi_{\beta}^{(\ell)},\psi_{\alpha}^{(L)})
=(T_{\ell\to L}^{(0)}E_\ell\bm{v})_\alpha.
\]
This proves the first identity.

Refinement also leaves the weak derivative of $v_h\in H_0^1(\Omega)$ unchanged. Since $\partial_qv_h\in\mathcal Q_\ell$, its exact coarse-level expansion in \eqref{bpxmapaction} similarly yields
\[
(C_L^{(q)}P_\ell\bm{v})_\alpha
=(\partial_qv_h,\psi_{\alpha}^{(L)})
=\sum_\beta(C_\ell^{(q)}\bm{v})_\beta
(\psi_{\beta}^{(\ell)},\psi_{\alpha}^{(L)})
=(T_{\ell\to L}^{(0)}C_\ell^{(q)}\bm{v})_\alpha,
\qquad 1\le q\le d.
\]
Stacking the $d$ component identities proves $C_LP_\ell\bm{v}=T_{\ell\to L}C_\ell\bm{v}$. Since $\bm{v}$ was arbitrary, both matrix identities follow. Finally, the two orthonormal representations of any $w\in\mathcal Q_\ell$ have the same coefficient norm $\|w\|_{L^2}$. Hence $(T_{\ell\to L}^{(0)})^TT_{\ell\to L}^{(0)}=I$, and its componentwise extension is also an isometry.
\end{proof}
\begin{remark}[Gradient transfers in the finite difference and finite element constructions]\label{rem:fdmfemtransfer}
The difference lies in the representation of the gradient coefficients. In the finite difference construction of \cite[Section~4.2.2]{JinLiuMaYu2026BPXrevised}, differentiation leaves nodal values in the other coordinate directions. For example, in two dimensions the first directional difference is $D_\ell^{(1)}=D_\ell^{\mathrm{1d}}\otimes I$, and, with $r=2^{L-\ell}$,
\[
D_L^{(1)}P_\ell
=\frac1r\bigl(E_{\ell\to L}^{\mathrm{1d}}\otimes P_\ell^{\mathrm{1d}}\bigr)D_\ell^{(1)}.
\]
Here $E_{\ell\to L}^{\mathrm{1d}}$ is the edge-replication matrix of that reference: it copies coarse differences along the differentiated direction. The other direction still requires nodal interpolation by $P_\ell^{\mathrm{1d}}$, which must therefore be included in the transfer encoding.

In the present finite element construction, the undifferentiated direction is also converted to local $L^2$-orthonormal coefficients. Thus
\[
C_\ell^{(1)}=C_\ell^{\mathrm{1d}}\otimes E_\ell^{\mathrm{1d}},\qquad
C_L^{(1)}P_\ell
=\bigl(T_{\ell\to L}^{\mathrm{1d}}\otimes T_{\ell\to L}^{\mathrm{1d}}\bigr)C_\ell^{(1)},
\]
where $T_{\ell\to L}^{\mathrm{1d}}$ transfers one-dimensional local orthonormal coefficients, as constructed below. Both directions now use orthonormal coefficient transfer, so a separate encoding of $P_\ell^{\mathrm{1d}}$ is unnecessary. Its action is retained through the commuting identity. This distinction concerns the chosen gradient representations; replacing the finite difference transfer by the present one would also require compatible gradient factors whose Gram matrix remains the original difference operator.
\end{remark}

The commuting identities explain how to encode the level contribution without introducing the fine-mesh derivative scale $h^{-1}$. Since $S_\ell=\omega_\ell P_\ell M_\ell^{-1/2}$, they give
\[
C_LS_\ell=\omega_\ell T_{\ell\to L}C_\ell M_\ell^{-1/2},
\qquad
\sqrt c\,E_LS_\ell=\sqrt c\,\omega_\ell T_{\ell\to L}^{(0)}E_\ell M_\ell^{-1/2}.
\]
The usefulness of these identities can be understood from two perspectives.
\begin{enumerate}
\item \textbf{Moving differentiation to the coarse level.}
Applying the product rule to separately encoded factors on the left can introduce the fine-grid scale $1/h$. Indeed,
\[
C_LS_\ell=\omega_\ell(C_LM_L^{-1/2})
(M_L^{1/2}P_\ell M_\ell^{-1/2}),
\]
and the first factor has encoding normalization $\sqrt{12d}/h$. Even with normalization $1$ for the second factor, the product normalization is $\omega_\ell\sqrt{12d}/h$. On the right, differentiation is performed on level $\ell$, so the corresponding normalization is $\sqrt{12d}/h_\ell$. The weight $\omega_\ell$ balances this coarse-level scale, giving
\[
\frac{\omega_\ell\sqrt{12d}}{h_\ell}
=\frac{\sqrt{12d}}{\sqrt{1+ch_\ell^2}}
\le\sqrt{12d}.
\]
The transfer $T_{\ell\to L}$ contributes normalization $1$. Similarly, the reaction contribution has normalization $\sqrt c\,\omega_\ell\le1$.

\item \textbf{Transferring the block encoding without increasing its normalization.}
The factor $1/h$ above arises from encoding the left-hand factors separately; it is not an unavoidable normalization for the complete preconditioned map $C_LS_\ell$. The commuting identity provides a way to encode this complete map through the coarse-level matrix $\omega_\ell C_\ell M_\ell^{-1/2}$. Since $T_{\ell\to L}^TT_{\ell\to L}=I$, we also have
\[
\omega_\ell C_\ell M_\ell^{-1/2}
=T_{\ell\to L}^T(C_LS_\ell).
\]
Lemma~\ref{lem:bpxtransferencoding} supplies exact encodings of $T_{\ell\to L}$ and its transpose with normalization $1$. Consequently, a generalized $(\alpha,\varepsilon)$-block-encoding of either $C_LS_\ell$ or $\omega_\ell C_\ell M_\ell^{-1/2}$ yields an encoding of the other with the same parameters, by the product rule. We can therefore construct the encoding on the coarse level, using the one-dimensional maps and tensor products in Lemma~\ref{lem:bpxnormalizedencoding}, and transfer it to the finest level without losing its favorable normalization. The same argument applies to the value map and $T_{\ell\to L}^{(0)}$.
\end{enumerate}
Combining the gradient and reaction contributions gives the mesh-independent energy normalization in \eqref{levelnormalizations}.

The transfer circuit is built from the one-dimensional subdivision matrix used in \cite[Section~6.1]{DeimlPeterseim2025QFEM}. Its derivation below accounts explicitly for the change in the local basis normalization.

\begin{lemma}[Block encodings of the interlevel transfers]\label{lem:bpxtransferencoding}
For $1\le\ell\le L$, the matrices $T_{\ell\to L}^{(0)}$ and $T_{\ell\to L}$ admit generalized block encodings with normalization $1$ in the ideal arithmetic model. Their implementations use a number of gates polynomial in $d$ and $L$.
\end{lemma}
\begin{proof}
We first consider one subdivision in one dimension. On a parent interval $I_i^{(\ell)}$, write $\varphi=(x-x_i)/h_\ell$ and let $w_0,w_1$ be the coefficients of a linear function in the local orthonormal basis:
\[
w(x)=\frac{w_0}{\sqrt{h_\ell}}
+\sqrt{\frac3{h_\ell}}\,w_1(2\varphi-1).
\]
The left and right halves are $[x_i,x_i+h_\ell/2]$ and $[x_i+h_\ell/2,x_i+h_\ell]$, respectively, and both have length $h_{\ell+1}=h_\ell/2$. The local coordinate on each half measures the distance from its own left endpoint, divided by its own length. Thus
\[
\varphi_-=
\frac{x-x_i}{h_\ell/2}=2\varphi,
\qquad
\varphi_+=
\frac{x-(x_i+h_\ell/2)}{h_\ell/2}=2\varphi-1.
\]
Each coordinate maps its corresponding half to $[0,1]$: it equals $0$ at that half's left endpoint and $1$ at its right endpoint. On the reference interval $[0,1]$, the orthonormal basis functions are $1$ and $\sqrt3(2\varphi_\pm-1)$. Mapping them to a half-interval of length $h_{\ell+1}$ requires the additional factor $h_{\ell+1}^{-1/2}$. Thus the orthonormal basis functions on the left and right halves are, respectively,
\[
\frac{1}{\sqrt{h_{\ell+1}}},\quad
\sqrt{\frac3{h_{\ell+1}}}(2\varphi_--1),
\qquad\text{and}\qquad
\frac{1}{\sqrt{h_{\ell+1}}},\quad
\sqrt{\frac3{h_{\ell+1}}}(2\varphi_+-1).
\]
Each basis function is extended by zero outside its corresponding half. To express the coarse-cell function in these bases, we use the identities
\[
2\varphi-1=\frac12(2\varphi_--1)-\frac12,
\qquad
2\varphi-1=\frac12(2\varphi_+-1)+\frac12.
\]
The coefficients $w_{-,0},w_{-,1}$ and $w_{+,0},w_{+,1}$ are defined by expanding the same function $w$ in the orthonormal bases on the left and right halves:
\[
\begin{aligned}
w(x)&=\frac{w_{-,0}}{\sqrt{h_{\ell+1}}}
+\sqrt{\frac3{h_{\ell+1}}}\,w_{-,1}(2\varphi_--1),
&&x\in[x_i,x_i+h_\ell/2],\\
w(x)&=\frac{w_{+,0}}{\sqrt{h_{\ell+1}}}
+\sqrt{\frac3{h_{\ell+1}}}\,w_{+,1}(2\varphi_+-1),
&&x\in[x_i+h_\ell/2,x_i+h_\ell].
\end{aligned}
\]
Substituting the coordinate identities into the coarse-cell expansion gives, on the left half,
\[
w(x)=\frac{w_0-\frac{\sqrt3}{2}w_1}{\sqrt{h_\ell}}
+\sqrt{\frac3{h_\ell}}\,\frac{w_1}{2}(2\varphi_--1),
\]
and, on the right half,
\[
w(x)=\frac{w_0+\frac{\sqrt3}{2}w_1}{\sqrt{h_\ell}}
+\sqrt{\frac3{h_\ell}}\,\frac{w_1}{2}(2\varphi_+-1).
\]
Comparing the constant and linear coefficients with the two normalized expansions, and using $h_{\ell+1}=h_\ell/2$, yields
\[
\begin{aligned}
w_{-,0}&=\frac1{\sqrt2}\Big(w_0-\frac{\sqrt3}{2}w_1\Big),
& w_{-,1}&=\frac{w_1}{2\sqrt2},\\
w_{+,0}&=\frac1{\sqrt2}\Big(w_0+\frac{\sqrt3}{2}w_1\Big),
& w_{+,1}&=\frac{w_1}{2\sqrt2}.
\end{aligned}
\]
Collecting these four scalar relations gives the local transfer matrix:
\[
\begin{bmatrix}w_{-,0}\\w_{-,1}\\w_{+,0}\\w_{+,1}\end{bmatrix}
=\frac1{\sqrt2}
\begin{bmatrix}
1&-\sqrt3/2\\
0&1/2\\
1&\sqrt3/2\\
0&1/2
\end{bmatrix}
\begin{bmatrix}w_0\\w_1\end{bmatrix}.
\]
Thus the first two rows give the constant and linear coefficients on the left half, and the last two rows give those on the right half. The factor $1/\sqrt2$ comes from the change in normalization when the cell length is halved. Both coefficient vectors represent exactly the same function on the parent interval.

Applying this local conversion independently to all $2^\ell$ parent intervals gives
\[
T_{\ell\to\ell+1}^{\mathrm{1d}}
=I^{\otimes\ell}\otimes\frac{1}{\sqrt2}
\begin{bmatrix}
1&-\sqrt3/2\\
0&1/2\\
1&\sqrt3/2\\
0&1/2
\end{bmatrix}.
\]
Here $I$ is the $2\times2$ identity, so $I^{\otimes\ell}$ acts on the parent-cell index and repeats the local conversion along the diagonal. Consequently, $T_{\ell\to\ell+1}^{\mathrm{1d}}$ maps $2^{\ell+1}$ coarse coefficients to $2^{\ell+2}$ fine coefficients.

The local $4\times2$ map, including the factor $1/\sqrt2$, has orthonormal columns. It can therefore be completed to a $4\times4$ unitary by initializing one additional qubit in $\ket0$. This realizes one subdivision with normalization $1$. Applying the same construction successively gives
\[
T_{\ell\to L}^{\mathrm{1d}}
=T_{L-1\to L}^{\mathrm{1d}}\cdots T_{\ell\to\ell+1}^{\mathrm{1d}},
\qquad
T_{\ell\to L}^{(0)}=(T_{\ell\to L}^{\mathrm{1d}})^{\otimes d}.
\]
For $\ell=L$, the transfer is the identity. Each subdivision appends one child-cell bit, and the tensor product applies this operation separately in each coordinate. All these maps have normalization $1$, so the product and tensor-product rules preserve this normalization.

Finally, $T_{\ell\to L}$ applies $T_{\ell\to L}^{(0)}$ independently to each gradient component. Its circuit leaves the direction register unchanged. Thus the direct sum over directions also has normalization $1$. There are $d(L-\ell)$ local subdivisions; the qubit permutations and the tests of active indices have polynomial cost in $d$ and $L$.
\end{proof}

\subsection{The one-dimensional inverse mass square root}

The factor $S_\ell$ contains $M_\ell^{-1/2}$. Its tensor-product structure reduces the construction to $(M_\ell^{\mathrm{1d}})^{-1/2}$, where $M_\ell^{\mathrm{1d}}=M_{h_\ell}$. We implement this one-dimensional map by a discrete sine transform, an eigenvalue-dependent flag rotation, and the inverse transform, before combining it with the local value and derivative maps.

On the $2^\ell-1$ interior nodes, the mass matrix is
\[
M_\ell^{\mathrm{1d}}=\frac{h_\ell}{6}\operatorname{tridiag}(1,4,1).
\]
Its eigenvectors are the discrete sine vectors. Indeed, with zero values at the two boundary nodes, the identity
\[
\sin\frac{\pi(i-1)j}{2^\ell}
+\sin\frac{\pi(i+1)j}{2^\ell}
=2\cos\frac{\pi j}{2^\ell}\sin\frac{\pi ij}{2^\ell}
\]
shows that multiplying a sine vector by $M_\ell^{\mathrm{1d}}$ multiplies it by $h_\ell(2+\cos(\pi j/2^\ell))/3$. Let $\mathcal F_\ell$ denote the discrete sine transform with entries
\[
(\mathcal F_\ell)_{ij}
=\sqrt{\frac{2}{2^\ell}}\sin\frac{\pi ij}{2^\ell},
\qquad 1\le i,j\le2^\ell-1.
\]
These sine vectors are orthonormal. We represent the interior indices by the computational states $\ket{1},\cdots,\ket{2^\ell-1}$ of an $\ell$-qubit register and extend $\mathcal F_\ell$ by fixing the unused state $\ket0$. On the active subspace spanned by the interior indices,
\begin{equation}\label{bpxmassspectrum}
\begin{aligned}
M_\ell^{\mathrm{1d}}&=\mathcal F_\ell\diag_j(\mu_{j}^{(\ell)})\mathcal F_\ell^\dagger,\\
\mu_{j}^{(\ell)}&=\frac{h_\ell}{3}\Big(2+\cos\frac{\pi j}{2^\ell}\Big),
\qquad 1\le j\le2^\ell-1.
\end{aligned}
\end{equation}
Consequently, its inverse square root is obtained by replacing each eigenvalue by its inverse square root:
\[
(M_\ell^{\mathrm{1d}})^{-1/2}
=\mathcal F_\ell\diag_j\bigl((\mu_j^{(\ell)})^{-1/2}\bigr)\mathcal F_\ell^\dagger.
\]
The bounds $h_\ell/3\le\mu_j^{(\ell)}\le h_\ell$ imply
$\|(M_\ell^{\mathrm{1d}})^{-1/2}\|\le\sqrt{3/h_\ell}$.
We therefore choose the block-encoding normalization
$\alpha_{(M_\ell^{\mathrm{1d}})^{-1/2}}=\sqrt{3/h_\ell}$.
With this choice, all the diagonal entries to be encoded lie between $1/\sqrt3$ and $1$.

We now implement this diagonal action using a flag qubit initialized in $\ket0$. First apply $\mathcal F_\ell^\dagger$ to express the input in the sine basis. For each eigenvalue index $j$, reversibly compute the angle
\[
\theta_j^{(\ell)}
=2\arccos\Big(\frac{1}{\sqrt{2+\cos(\pi j/2^\ell)}}\Big)
\]
in an arithmetic register and apply the controlled rotation $R_y(\theta_j^{(\ell)})$ to the flag. Its amplitude on $\ket0$ is
\begin{equation}\label{bpxmassrotation}
\frac{1}{\alpha_{(M_\ell^{\mathrm{1d}})^{-1/2}}\sqrt{\mu_{j}^{(\ell)}}}
=\frac{1}{\sqrt{2+\cos(\pi j/2^\ell)}},
\qquad \alpha_{(M_\ell^{\mathrm{1d}})^{-1/2}}=\sqrt{3/h_\ell}.
\end{equation}
Thus, after returning the arithmetic register to zero, the resulting controlled rotation $\mathcal U_\ell^{\mathrm{rot}}$ acts as
\[
\ket j\ket0\longmapsto\ket j\Big(
\frac{\ket0}{\sqrt{2+\cos(\pi j/2^\ell)}}
+\sqrt{1-\frac{1}{2+\cos(\pi j/2^\ell)}}\ket1
\Big).
\]
For the unused index $j=0$, we instead flip the flag to $\ket1$, so that this index has zero amplitude in the successful branch. Finally, apply $\mathcal F_\ell$ to return to nodal coordinates. Suppressing the arithmetic registers, which have been returned to zero, the ideal unitary is
\[
\mathcal U_{(M_\ell^{\mathrm{1d}})^{-1/2}}
=(\mathcal F_\ell\otimes I)\mathcal U_\ell^{\mathrm{rot}}(\mathcal F_\ell^\dagger\otimes I).
\]
For any normalized input $\ket{\bm{v}}$ supported on the interior nodes, its output has the form
\[
\mathcal U_{(M_\ell^{\mathrm{1d}})^{-1/2}}
\bigl(\ket{\bm{v}}\ket0\bigr)
=\frac{(M_\ell^{\mathrm{1d}})^{-1/2}\ket{\bm{v}}}
{\sqrt{3/h_\ell}}\ket0
+\ket{\mathrm{fail}_{\bm{v}}}\ket1,
\]
where $\ket{\mathrm{fail}_{\bm{v}}}$ denotes the possibly unnormalized vector in the other flag branch. Projecting onto flag state $\ket0$ therefore gives
\begin{equation}\label{bpxmassencoding}
(I^{\otimes \ell}\otimes\bra0)
\mathcal U_{(M_\ell^{\mathrm{1d}})^{-1/2}}
(I^{\otimes \ell}\otimes\ket0)
=\frac{(M_\ell^{\mathrm{1d}})^{-1/2}}{\sqrt{3/h_\ell}}.
\end{equation}
Here the matrix on the right is extended by zero on the unused index. The projection identifies the encoded block; no measurement is required when this unitary is composed with the value and derivative encodings below.

The sine transform can be implemented using a quantum Fourier transform \cite[Chapter~5]{Nielsen2010} on an odd extension of the nodal vector. Since $2^\ell$ is a power of two, the extended transform uses $\ell+1$ qubits. The scalar functions used in the rotation are evaluated on $2+\cos(\pi j/2^\ell)\in[1,3]$, so their arguments remain bounded away from the singularity of the inverse square root. The transforms, reversible arithmetic, and rotations can therefore be implemented with a gate cost polynomial in $\ell$ and the number of precision bits. If the implemented unitary differs from the ideal one by at most $\zeta$ in operator norm, its successful block differs by at most $\zeta$, and the corresponding absolute error in the inverse mass square root is at most $\sqrt{3/h_\ell}\,\zeta$. Thus an absolute matrix tolerance $\delta$ is achieved by choosing $\zeta\le\delta\sqrt{h_\ell/3}$. This precision requirement is retained when the inverse mass encoding is composed with the local value and derivative maps.

\subsection{Block encodings of the normalized value and gradient maps}

We now encode $E_\ell M_\ell^{-1/2}$ and $C_\ell M_\ell^{-1/2}$. The key step is to reduce the normalization of the one-dimensional value encoding to $1$ before taking tensor products. Only the derivative factor then contributes a mesh-dependent normalization.

\begin{lemma}[Normalized value and gradient encodings]\label{lem:bpxnormalizedencoding}
For $1\le\ell\le L$, the matrices $E_\ell M_\ell^{-1/2}$ and $C_\ell M_\ell^{-1/2}$ admit generalized block encodings with normalizations $1$ and $\sqrt{12d}/h_\ell$, respectively, in the ideal arithmetic model.
\end{lemma}
\begin{proof}
\noindent\textbf{Step 1: One-dimensional value and derivative matrices.}
For $\bm{v}=(v_1,\cdots,v_{2^\ell-1})^T$, set $v_0=v_{2^\ell}=0$ and define the two endpoint maps by
\[
R_{\ell,0}\bm{v}=\sum_{i=0}^{2^\ell-1}v_i\ket{i},
\qquad
R_{\ell,1}\bm{v}=\sum_{i=0}^{2^\ell-1}v_{i+1}\ket{i}.
\]
Both maps send distinct interior nodes to distinct cell indices, so $R_{\ell,r}^TR_{\ell,r}=I$ for $r=0,1$. They are implemented by index shifts with zero padding and admit exact generalized $(1,0)$-block-encodings $U_0,U_1$ with common input and output embeddings. If $H$ denotes the Hadamard gate, the selector circuits
\[
U_\pm=(H\otimes I)
\Big(\ket0\bra0\otimes U_1\ \pm\ \ket1\bra1\otimes U_0\Big)
(H\otimes I)
\]
have successful blocks $(R_{\ell,1}\pm R_{\ell,0})/2$. Thus the endpoint sum and difference both admit generalized $(2,0)$-block-encodings.

By \eqref{bpxlocalmaps}, with the output ordered as $\ket{i}\ket{\nu}$,
\[
\begin{aligned}
E_\ell^{\mathrm{1d}}
&=\frac{\sqrt{h_\ell}}2(R_{\ell,1}+R_{\ell,0})\otimes\ket0
+\sqrt{\frac{h_\ell}{12}}(R_{\ell,1}-R_{\ell,0})\otimes\ket1,\\
C_\ell^{\mathrm{1d}}
&=\frac1{\sqrt{h_\ell}}(R_{\ell,1}-R_{\ell,0})\otimes\ket0.
\end{aligned}
\]
The two value components therefore have normalizations $\sqrt{h_\ell}$ and $\sqrt{h_\ell/3}$. The concatenation rule in Lemma~\ref{lem:arithmeticBE} gives
\[
\alpha_{E_\ell^{\mathrm{1d}}}
=\sqrt{h_\ell+\frac{h_\ell}3}=\sqrt{\frac{4h_\ell}3},
\qquad
\alpha_{C_\ell^{\mathrm{1d}}}=\frac2{\sqrt{h_\ell}}.
\]
For example, the value encoding uses the selector state $(\sqrt3\ket0+\ket1)/2$ and applies $U_+$ or $U_-$, retaining the selector as the local mode index. Its successful output is $E_\ell^{\mathrm{1d}}\bm{v}/\sqrt{4h_\ell/3}$.

\noindent\textbf{Step 2: Composition with the inverse mass square root.}
The inverse mass encoding in \eqref{bpxmassencoding} has normalization $\sqrt{3/h_\ell}$. Applying the product rule with separate flags yields
\[
\begin{aligned}
\alpha_{E_\ell^{\mathrm{1d}}(M_\ell^{\mathrm{1d}})^{-1/2}}
&=\sqrt{\frac{4h_\ell}3}\sqrt{\frac3{h_\ell}}=2,\\
\alpha_{C_\ell^{\mathrm{1d}}(M_\ell^{\mathrm{1d}})^{-1/2}}
&=\frac2{\sqrt{h_\ell}}\sqrt{\frac3{h_\ell}}
=\frac{\sqrt{12}}{h_\ell}.
\end{aligned}
\]
The first normalization can be reduced to $1$. Indeed, Lemma~\ref{lem:bpx1dfactors} gives
\[
\Big(E_\ell^{\mathrm{1d}}(M_\ell^{\mathrm{1d}})^{-1/2}\Big)^T
\Big(E_\ell^{\mathrm{1d}}(M_\ell^{\mathrm{1d}})^{-1/2}\Big)
=(M_\ell^{\mathrm{1d}})^{-1/2}M_\ell^{\mathrm{1d}}(M_\ell^{\mathrm{1d}})^{-1/2}=I.
\]
Let $U$ denote the product encoding of this value map, and write its successful block as
\[
F:=\Pi_{\mathrm{out}}^\dagger U\Pi_{\mathrm{in}}
=\tfrac12 E_\ell^{\mathrm{1d}}(M_\ell^{\mathrm{1d}})^{-1/2},
\qquad F^\dagger F=\tfrac14I.
\]
Using reflections about the active input and successful output subspaces, define
\[
\begin{aligned}
R_{\mathrm{in}}&=2\Pi_{\mathrm{in}}\Pi_{\mathrm{in}}^\dagger-I,
\qquad R_{\mathrm{out}}=2\Pi_{\mathrm{out}}\Pi_{\mathrm{out}}^\dagger-I,\\
U_{\mathrm{amp}}&=-UR_{\mathrm{in}}U^\dagger R_{\mathrm{out}}U.
\end{aligned}
\]
Its successful block is
\[
\Pi_{\mathrm{out}}^\dagger U_{\mathrm{amp}}\Pi_{\mathrm{in}}
=3F-4FF^\dagger F
=2F=E_\ell^{\mathrm{1d}}(M_\ell^{\mathrm{1d}})^{-1/2}.
\]
Thus one oblivious amplitude-amplification step, using three calls to $U$ or $U^\dagger$, gives a generalized $(1,0)$-block-encoding of the normalized value map. The derivative encoding retains normalization $\sqrt{12}/h_\ell$. Finite-precision implementations can be treated by robust oblivious amplitude amplification \cite[Theorem~28]{Gilyen2019QSVD}.

\noindent\textbf{Step 3: Tensor products and the gradient components.}
Since $M_\ell=(M_\ell^{\mathrm{1d}})^{\otimes d}$, we have
\begin{equation}\label{bpxnormalizedvalue}
E_\ell M_\ell^{-1/2}
=\Big[E_\ell^{\mathrm{1d}}(M_\ell^{\mathrm{1d}})^{-1/2}\Big]^{\otimes d},
\qquad \alpha_{E_\ell M_\ell^{-1/2}}=1^d=1.
\end{equation}
Here each one-dimensional value encoding is amplified before forming the tensor product. For the $q$th gradient component,
\[
\begin{aligned}
C_\ell^{(q)}M_\ell^{-1/2}
={}&\Big[E_\ell^{\mathrm{1d}}(M_\ell^{\mathrm{1d}})^{-1/2}\Big]^{\otimes(q-1)}\\
&\otimes\Big[C_\ell^{\mathrm{1d}}(M_\ell^{\mathrm{1d}})^{-1/2}\Big]
\otimes\Big[E_\ell^{\mathrm{1d}}(M_\ell^{\mathrm{1d}})^{-1/2}\Big]^{\otimes(d-q)},\\
\alpha_{C_\ell^{(q)}M_\ell^{-1/2}}
={}&1^{q-1}\frac{\sqrt{12}}{h_\ell}1^{d-q}
=\frac{\sqrt{12}}{h_\ell}.
\end{aligned}
\]
Prepare the direction state $d^{-1/2}\sum_{q=1}^d\ket{q-1}$ and apply the component encodings controlled on this register. For an input $\bm{v}$, retaining the direction label gives the successful output
\[
\frac1{\sqrt d}\sum_{q=1}^d\ket{q-1}\otimes
\frac{C_\ell^{(q)}M_\ell^{-1/2}\bm{v}}{\sqrt{12}/h_\ell}
=\frac{C_\ell M_\ell^{-1/2}\bm{v}}{\sqrt{12d}/h_\ell},
\]
where the right-hand side uses the stacked ordering of the gradient components. Hence
\begin{equation}\label{bpxnormalizedgradient}
\alpha_{C_\ell M_\ell^{-1/2}}
=\Big(\sum_{q=1}^d\frac{12}{h_\ell^2}\Big)^{1/2}
=\frac{\sqrt{12d}}{h_\ell}.
\end{equation}
\end{proof}

\subsection{Block encoding of the BPX system}

The level encodings yield the BPX matrix $A_S=S^TK_{h,d}S$ and an $L^2$ representation of the finite element solution. Define
\begin{equation}\label{compositemaps}
\begin{aligned}
\mathcal E&=E_LS=[\mathcal E_1\ \cdots\ \mathcal E_L],\\
\mathcal C&=\begin{bmatrix}C_LS\\\sqrt c\,E_LS\end{bmatrix}
=[\mathcal C_1\ \cdots\ \mathcal C_L].
\end{aligned}
\end{equation}
We call $\mathcal E=E_LS$ the $L^2$ solution representation map. For a solution $\bm{z}$ of \eqref{bpxsystem}, the coefficient transformations are
\[
\underbrace{\bm{z}}_{\text{multilevel coefficients}}
\xrightarrow{\ S\ }
\underbrace{\bm{u}_h=S\bm{z}}_{\text{nodal coefficients}}
\xrightarrow{\ E_L\ }
\underbrace{E_L\bm{u}_h=\mathcal E\bm{z}}_{\text{$L^2$-orthonormal coefficients}}.
\]
The last two vectors represent the same finite element function $u_h$: the former uses the nodal basis, and the latter uses the finest-level cellwise $L^2$-orthonormal basis $\{\psi_{\bm{m},\bm{\nu}}^{(L)}\}$. This orthonormal basis spans the space of cellwise multilinear functions containing $V_h$; it is not obtained merely by rescaling the nodal basis. We use its coefficients as the solution representation, with $\|\mathcal E\bm{z}\|_2=\|u_h\|_{L^2}$. The arrows describe coefficient identities: the circuit implements the composite map $\mathcal E$ without first preparing an intermediate nodal coefficient state. Thus the local orthonormal basis serves both the block-encoding construction and the final output. The map $\mathcal C$ additionally represents the gradient and multiplies the value coefficients by $\sqrt c$, so its Gram matrix gives the preconditioned reaction--diffusion system.

\noindent\textbf{Level encodings.}
The commuting identities give
\begin{equation}\label{compositelevels}
\begin{aligned}
\mathcal E_\ell&=\omega_\ell T_{\ell\to L}^{(0)}E_\ell M_\ell^{-1/2},\\
\mathcal C_\ell&=\omega_\ell
\begin{bmatrix}
T_{\ell\to L}C_\ell M_\ell^{-1/2}\\
\sqrt c\,T_{\ell\to L}^{(0)}E_\ell M_\ell^{-1/2}
\end{bmatrix},\\
A_S&=\mathcal C^T\mathcal C
=S^T(A_{h,d}+cM_{h,d})S.
\end{aligned}
\end{equation}
By Lemmas~\ref{lem:bpxtransferencoding} and \ref{lem:bpxnormalizedencoding}, the gradient and value maps inside these expressions admit encodings with normalizations $\sqrt{12d}/h_\ell$ and $1$, respectively. To combine the two parts of $\mathcal C_\ell$, prepare
\[
\frac{(\sqrt{12d}/h_\ell)\ket0+\sqrt c\ket1}
{\sqrt{12d\,h_\ell^{-2}+c}}
\]
and apply the gradient or value encoding controlled on this selector, followed by the appropriate transfer. For an input $\bm{z}_\ell$, the successful output is
\[
\begin{aligned}
&\frac1{\sqrt{12d\,h_\ell^{-2}+c}}
\Big(\ket0\otimes T_{\ell\to L}C_\ell M_\ell^{-1/2}\bm{z}_\ell
+\sqrt c\ket1\otimes T_{\ell\to L}^{(0)}E_\ell M_\ell^{-1/2}\bm{z}_\ell\Big)\\
&\hspace{2em}=\frac{\mathcal C_\ell\bm{z}_\ell}
{\omega_\ell\sqrt{12d\,h_\ell^{-2}+c}},
\end{aligned}
\]
where the selector distinguishes the gradient and value output spaces. Consequently,
\begin{equation}\label{levelnormalizations}
\alpha_{\mathcal E_\ell}=\omega_\ell\le h_\ell,\qquad
\alpha_{\mathcal C_\ell}^2
=\omega_\ell^2(12d\,h_\ell^{-2}+c)
=\frac{12d\,h_\ell^{-2}+c}{h_\ell^{-2}+c}\le12d.
\end{equation}
At $c=0$, only the gradient branch is present. Evaluation of the selector amplitudes and weights is included in the arithmetic gate cost.

\begin{theorem}[Block encodings of the reaction--diffusion BPX system]\label{thm:encoding}
In the ideal arithmetic model, $\mathcal C$, $A_S$, and $\mathcal E$ admit generalized $(\alpha_{\mathcal C},0)$-, $(\alpha_{A_S},0)$-, and $(\alpha_{\mathcal E},0)$-block-encodings, respectively, with
\begin{equation}\label{normalizations}
\alpha_{\mathcal C}=\sqrt{12dL},\qquad
\alpha_{A_S}=12dL,\qquad
\alpha_{\mathcal E}=\frac1{\sqrt3}.
\end{equation}
Moreover, $\|\mathcal E\|\le1/\sqrt3$. These bounds are uniform in $c\ge0$. The circuits use polynomially many gates in $d$, $L$, and the arithmetic bit lengths, including the representation of $c$.
\end{theorem}
\begin{proof}
\noindent\textbf{Step 1: Combining the level energy maps.}
An additional flag rotation with successful amplitude $\alpha_{\mathcal C_\ell}/\sqrt{12d}$ raises each level normalization to $\sqrt{12d}$. Apply these encodings controlled on the level register of
\[
\bm{z}=\sum_{\ell=1}^L\ket{\ell-1}\otimes\bm{z}_\ell,
\]
and project that register onto $L^{-1/2}\sum_{\ell=1}^L\ket{\ell-1}$. The successful output is
\[
\frac1{\sqrt L}\sum_{\ell=1}^L
\frac{\mathcal C_\ell\bm{z}_\ell}{\sqrt{12d}}
=\frac{\mathcal C\bm{z}}{\sqrt{12dL}}.
\]
This projection is implemented coherently by inverse state preparation and a success flag. It gives an encoding $U_{\mathcal C}$ satisfying
\[
\Pi_{\mathrm{out}}^\dagger U_{\mathcal C}\Pi_{\mathrm{in}}
=\frac{\mathcal C}{\sqrt{12dL}}.
\]

\noindent\textbf{Step 2: Encoding the BPX matrix.}
Use the adjoint encoding for $\mathcal C^T$ and the product rule with separate flags and the intermediate success projection. The resulting successful block is
\[
\frac{\mathcal C^T}{\sqrt{12dL}}
\frac{\mathcal C}{\sqrt{12dL}}
=\frac{S^T(C_L^TC_L+cE_L^TE_L)S}{12dL}
=\frac{A_S}{12dL}.
\]
Thus $\alpha_{A_S}=12dL$. The intermediate projection is essential: the unprojected product $U_{\mathcal C}^\dagger U_{\mathcal C}$ would simply be the identity.

\noindent\textbf{Step 3: Encoding the $L^2$ solution representation map.}
For $\mathcal E_\ell$, a flag rotation with successful amplitude $\omega_\ell/h_\ell$ gives normalization $h_\ell$. Apply the level-controlled encodings and project the level register onto
\[
\frac{\sum_{\ell=1}^Lh_\ell\ket{\ell-1}}
{\big(\sum_{\ell=1}^Lh_\ell^2\big)^{1/2}}.
\]
The successful output is
\[
\frac1{\big(\sum_{\ell=1}^Lh_\ell^2\big)^{1/2}}
\sum_{\ell=1}^Lh_\ell\frac{\mathcal E_\ell\bm{z}_\ell}{h_\ell}
=\frac{\mathcal E\bm{z}}
{\big(\sum_{\ell=1}^Lh_\ell^2\big)^{1/2}}.
\]
Since $h_\ell=2^{-\ell}$,
\[
\sum_{\ell=1}^Lh_\ell^2=\frac{1-4^{-L}}3.
\]
A final flag rotation with successful amplitude $\sqrt{1-4^{-L}}$ therefore raises the normalization to $1/\sqrt3$. Also, the isometry identities give $\|\mathcal E_\ell\bm{z}_\ell\|_2=\omega_\ell\|\bm{z}_\ell\|_2$, and hence
\[
\|\mathcal E\bm{z}\|_2
\le\sum_{\ell=1}^L\omega_\ell\|\bm{z}_\ell\|_2
\le\Big(\sum_{\ell=1}^Lh_\ell^2\Big)^{1/2}\|\bm{z}\|_2
\le\frac1{\sqrt3}\|\bm{z}\|_2.
\]
The local circuits, controlled level selection, flag rotations, and level-state preparations have polynomial gate costs in the stated parameters.
\end{proof}

Amplifying the one-dimensional normalized value map before taking tensor products prevents a normalization loss in the nondifferentiated coordinates. Stacking the $d$ gradient components then gives the factor $\sqrt d$, and the final matrix normalization is $12dL$. In particular, no factor $3^{d-1}$ appears. The nodal BPX construction in \cite{DeimlPeterseim2025QFEM} also admits favorable encodings; the constants derived here apply to \eqref{bpxfactor}.


\section{Error and complexity analysis}\label{sec:accuracy}

We assume direct access to the preconditioned input states and analyze preparation of the solution state in the local $L^2$-orthonormal basis and estimation of a bounded linear functional.

\subsection{Solution-state preparation}\label{sec:recovery}

\noindent\textbf{Preconditioned right-hand side and oracle access.}
Following the input model of \cite[Section~5.1]{JinLiuMaYu2026BPXrevised}, define
\begin{equation}\label{inputrelations}
\bm{b}_S=S^T\bm{f}_{h,d},\qquad A_S\bm{z}=\bm{b}_S,\qquad \bm{z}=A_S^+\bm{b}_S.
\end{equation}
For $\bm{b}_S\ne0$, assume a state-preparation oracle
\[
O_{Sb}\ket0=\ket{\bm{b}_S}=\frac{\bm{b}_S}{\|\bm{b}_S\|_2},
\]
together with the supplied norm $\|\bm{b}_S\|_2$. The query bounds count calls to this oracle and to the multilevel block encodings, including their inverses and controlled versions. They do not include the construction of $O_{Sb}$ from the original load vector or from classical data. In particular, applying $S^T$ to a separately prepared load state is not part of the assumed input procedure. If $\bm{b}_S=0$, full row rank of $S$ implies $\bm{f}_{h,d}=0$, and the discrete solution is zero.

For later norm estimates, let $\bm{f}_h^{L^2}$ denote the coefficients of $\Pi_hf$ in the finest-level orthonormal basis. Since $E_L^T\bm{f}_h^{L^2}=\bm{f}_{h,d}$ and $\mathcal E=E_LS$, we have
\begin{equation}\label{preconditionedinputnorm}
\bm{b}_S=\mathcal E^T\bm{f}_h^{L^2},\qquad
\|\bm{b}_S\|_2\le\frac1{\sqrt3}\|\Pi_hf\|_{L^2}\le\frac1{\sqrt3}\|f\|_{L^2}.
\end{equation}
This identity bounds the supplied input norm; it is not an additional input-state preparation step.

\noindent\textbf{Approximation of the inverse.}
Theorems~\ref{thm:bpxspectrum} and \ref{thm:encoding} imply
\[
\sigma(A_S)\setminus\{0\}\subset[1/L,12dL],\qquad
\sigma\big(A_S/(12dL)\big)\setminus\{0\}\subset[1/(12dL^2),1].
\]
Quantum singular value transformation \cite{Gilyen2019QSVD} gives an odd inverse polynomial, written in the original matrix scale as $p_\delta$, satisfying
\begin{equation}\label{inverseapproximation}
\|p_\delta(A_S)-A_S^+\|\le\delta,\qquad p_\delta(0)=0.
\end{equation}
Its generalized $(2L,\,0)$-block-encoding in ideal arithmetic uses
\begin{equation}\label{inversequeries}
\mathcal{O}\bigl(dL^2\log(2+L/\delta)\bigr)
\end{equation}
queries to the encoding of $A_S$ and its inverse. Here $2L$ is the encoding normalization of the polynomial approximation, while $\delta$ measures its error relative to $A_S^+$. The condition $p_\delta(0)=0$ preserves the nullspace of the redundant BPX system.

Given $\varepsilon>0$, we seek a represented function $u_h^q$ satisfying
\begin{equation}\label{targetaccuracy}
\|u-u_h^q\|_{L^2}\le\varepsilon.
\end{equation}
For solution-state preparation, we consider the accuracy regime $0<\varepsilon\le\|u\|_{L^2}/2$. The solution norm enters the complexity analysis but is not required as an input to the algorithm. We allocate $\varepsilon/3$ each to finite element discretization, inverse approximation, and circuit implementation. By Lemma~\ref{lem:feerror}, a sufficient mesh condition is $3h^2\|u\|_{H^2}\le\varepsilon/3$.

Starting from the supplied preconditioned right-hand side, the coefficient transformations are
\[
\bm{b}_S\xrightarrow{\ A_S^+\ }\bm{z}
\xrightarrow{\ S\ }\bm{u}_h
\xrightarrow{\ E_L\ }E_L\bm{u}_h,
\qquad \mathcal E\bm{z}=E_L\bm{u}_h.
\]
The circuit combines the inverse encoding with the encoding of $\mathcal E$, using separate flags. With the exact pseudoinverse, the joint successful branch would be
\begin{equation}\label{solutionbranch}
\frac{\mathcal E A_S^+\bm{b}_S}{2L\alpha_{\mathcal E}\|\bm{b}_S\|_2}
=\frac{E_L\bm{u}_h}{(2L/\sqrt3)\|\bm{b}_S\|_2},
\qquad \alpha_{\mathcal E}=\frac1{\sqrt3}.
\end{equation}
Replacing $A_S^+$ by $p_\delta(A_S)$ defines
\begin{equation}\label{representative}
\bm{u}_h^q=Sp_\delta(A_S)\bm{b}_S,\qquad
u_h^q=\sum_{\bm{j}}(\bm{u}_h^q)_{\bm{j}}\phi_{\bm{j}}^{(L)}.
\end{equation}
The vector $\bm{u}_h^q$ in \eqref{representative} defines the nodal expansion of the represented function, but the circuit outputs its coefficients in the local $L^2$-orthonormal basis. Specifically, the successful branch is $E_L\bm{u}_h^q/[(2L/\sqrt3)\|\bm{b}_S\|_2]$, and its normalized state is
\[
\frac{E_L\bm{u}_h^q}{\|E_L\bm{u}_h^q\|_2}
=\frac{E_L\bm{u}_h^q}{\|u_h^q\|_{L^2}}.
\]
Multiplication of the successful branch by the known scale $(2L/\sqrt3)\|\bm{b}_S\|_2$ defines the ``unnormalized'' representative $E_L\bm{u}_h^q$. The quotation marks distinguish this scaled branch from the normalized output state. Its error satisfies
\[
\|E_L\bm{u}_h-E_L\bm{u}_h^q\|_2=\|u_h-u_h^q\|_{L^2}.
\]
This identity lets us retain the orthonormal representation when imposing continuous accuracy, without converting the output to the nodal coefficient state $\bm{u}_h^q/\|\bm{u}_h^q\|_2$. The function error is measured using the scaled branch, without requiring the exact solution norm.

\begin{theorem}[Solution error and query complexity]\label{thm:complexity}
Assume the stated preconditioned input access. Let $0<\varepsilon\le\|u\|_{L^2}/2$, choose a mesh satisfying $\|u-u_h\|_{L^2}\le\varepsilon/3$, and set
\begin{equation}\label{deltachoice}
\delta=\min\Big\{1,\frac{\varepsilon}{\sqrt3\|\bm{b}_S\|_2}\Big\}.
\end{equation}
Then $\bm{b}_S\ne0$, and the ideal representative in \eqref{representative} satisfies $\|u-u_h^q\|_{L^2}\le2\varepsilon/3$. With an additional function error of at most $\varepsilon/3$ from circuit implementation, the represented output satisfies \eqref{targetaccuracy}. Its normalized coefficient state in the finest-level local $L^2$-orthonormal basis can be prepared with a success flag and constant success probability using expected
\begin{equation}\label{queriesratio}
\mathcal{O}\!\Big(dL^3\frac{\|\bm{b}_S\|_2}{\|u_h\|_{L^2}}\log(2+L/\delta)\Big)
\end{equation}
queries. In particular, a bound in terms of the continuous data is
\begin{equation}\label{queriesdata}
\mathcal{O}\!\Big(dL^3\frac{\|f\|_{L^2}}{\|u\|_{L^2}}\log\!\Big(2+\frac{L\|f\|_{L^2}}{\varepsilon}\Big)\Big),\qquad L=\log_2(1/h).
\end{equation}
For fixed $\|f\|_{L^2}/\|u\|_{L^2}$, this bound has no algebraic inverse-mesh factor and depends logarithmically on $1/\varepsilon$, apart from the choice of mesh required for finite element accuracy.
\end{theorem}
\begin{proof}
\noindent\textbf{Step 1: Continuous error from the inverse approximation.}
The discretization error and the restriction on $\varepsilon$ imply
\[
\|u_h\|_{L^2}\ge\|u\|_{L^2}-\|u-u_h\|_{L^2}
\ge\|u\|_{L^2}-\frac{\varepsilon}{3}\ge\frac56\|u\|_{L^2}>0.
\]
Consequently, $\bm{b}_S\ne0$. The orthonormal coefficient representation and $\|\mathcal E\|\le1/\sqrt3$ give
\begin{equation}\label{continuouserror}
\begin{aligned}
\|u_h-u_h^q\|_{L^2}
&=\|\mathcal E(A_S^+-p_\delta(A_S))\bm{b}_S\|_2\\
&\le\frac{\delta}{\sqrt3}\|\bm{b}_S\|_2\le\frac{\varepsilon}{3}.
\end{aligned}
\end{equation}
Adding the discretization error gives the ideal function error bound.

\noindent\textbf{Step 2: Size of the successful branch.}
Using \eqref{preconditionedinputnorm}, we obtain
\begin{equation}\label{solutionnorm}
\frac{\|\bm{b}_S\|_2}{\|u_h\|_{L^2}}
\le\frac{\|f\|_{L^2}}{\sqrt3(\|u\|_{L^2}-\varepsilon/3)}
\le\frac{6\|f\|_{L^2}}{5\sqrt3\|u\|_{L^2}}.
\end{equation}
Moreover, $\varepsilon/3\le\|u\|_{L^2}/6\le\|u_h\|_{L^2}/5$, so \eqref{continuouserror} implies
\[
\|u_h^q\|_{L^2}\ge\frac45\|u_h\|_{L^2},\qquad
\|u_h^q\|_{L^2}\ge\|u\|_{L^2}-\frac{2\varepsilon}{3}\ge\frac23\|u\|_{L^2}.
\]
The successful branch therefore has norm at least
\[
\frac{\sqrt3\|u_h^q\|_{L^2}}{2L\|\bm{b}_S\|_2}
\ge\frac{2\sqrt3\|u_h\|_{L^2}}{5L\|\bm{b}_S\|_2}.
\]
Thus the absolute accuracy requirement already ensures a sufficiently large branch; no inverse estimate involving $h^{-1}$ is needed.

\noindent\textbf{Step 3: Amplification cost.}
Amplitude amplification with a success test \cite{Brassard2002Amplitude} uses expected
\[
\mathcal{O}\!\Big(L\frac{\|\bm{b}_S\|_2}{\|u_h\|_{L^2}}\Big)
\]
calls to the response circuit and its inverse. Geometrically increasing amplification lengths give this expected order without knowing the success probability. Multiplying by \eqref{inversequeries} proves \eqref{queriesratio}. Also, $\Pi_hf=\mathcal A_hu_h$ and \eqref{generalizedspectrum} imply
\[
\|f\|_{L^2}\ge\|\Pi_hf\|_{L^2}\ge(c+4d)\|u_h\|_{L^2}>\varepsilon.
\]
Consequently, \eqref{deltachoice} and \eqref{preconditionedinputnorm} give
\[
\frac{L}{\delta}
=L\max\Big\{1,\frac{\sqrt3\|\bm{b}_S\|_2}{\varepsilon}\Big\}
\le\frac{L\|f\|_{L^2}}{\varepsilon}.
\]
Together with \eqref{solutionnorm}, this proves \eqref{queriesdata}.

\noindent\textbf{Step 4: Circuit implementation error.}
An operator error $\eta$ in the response unitary perturbs the scaled successful branch by at most $(2L/\sqrt3)\|\bm{b}_S\|_2\eta$. It suffices to choose
\begin{equation}\label{circuitprecision}
\eta\le\frac{\varepsilon}{2\sqrt3 L\|\bm{b}_S\|_2},
\end{equation}
which makes the corresponding function perturbation at most $\varepsilon/3$. Including both inverse and implementation errors, the implemented representative has norm at least
\[
\|u_h\|_{L^2}-\frac{2\varepsilon}{3}\ge\frac35\|u_h\|_{L^2},\qquad
\|u\|_{L^2}-\varepsilon\ge\frac12\|u\|_{L^2}.
\]
Its total continuous error is at most $\varepsilon$, and its successful branch retains the size required in Step~3. Primitive errors within the response circuit are allocated by telescoping to satisfy \eqref{circuitprecision}. Amplitude amplification uses this implemented unitary and its circuit adjoint, together with reflections about the input and success subspaces, and preserves the direction of its successful branch. Here the state-preparation circuit includes $O_{Sb}$: the input reflection is about the all-zero state, and the success reflection acts on the encoding flags. Both are exact in the circuit model. The representative is defined by the scaled branch before amplification, as in \eqref{representative}. The precision requirements change the elementary gate cost but not the stated query bounds.
\end{proof}

The ratio $\|f\|_{L^2}/\|u\|_{L^2}$ in \eqref{queriesdata} is part of the complexity estimate. For a fixed continuous problem with $u\ne0$, this ratio is independent of $h$, so the mesh dependence is through $L=\log_2(1/h)$. It may nevertheless grow for a family of increasingly oscillatory right-hand sides. Multiplying $f$ and $u$ by the same positive constant leaves this ratio unchanged. The condition $\varepsilon\le\|u\|_{L^2}/2$ specifies the accuracy regime of the theorem; neither the inverse tolerance nor the adaptive amplification procedure requires the exact solution norm.

Direct access to $\ket{\bm{b}_S}$ removes the cost of forming the preconditioned right-hand side. The conversion from multilevel coordinates to the original solution representation still incurs the amplification factor in \eqref{queriesratio}, now controlled by \eqref{solutionnorm}. If only $\ket{\bm{z}}$ is required, then $\|\bm{b}_S\|_2/\|\bm{z}\|_2\le12dL$, so its normalized-state preparation requires no condition relating the requested accuracy to $\|u\|_{L^2}$. The explicit tensor mass factor $3^{d-1}$ is absent from both constructions. Each query has the polynomial gate overhead stated in Theorem~\ref{thm:encoding}. At finite gate precision the represented function may lie in $\mathcal Q_L$ rather than exactly in $V_h$; its error is still measured in $L^2$. Estimating its norm or reading out all nodal coefficients incurs additional costs.

\subsection{Estimation of linear functionals}\label{subsec:linearfunctionals}

For a prescribed real-valued $g\in L^2(\Omega)$, let
\begin{equation}\label{linearfunctional}
J(u)=(g,u),\qquad J(u_h)=(g,u_h).
\end{equation}
Define the load vector $(\bm{g}_{h,d})_{\bm{j}}=(g,\phi_{\bm{j}}^{(L)})$ and the preconditioned test vector $\bm{c}_S=S^T\bm{g}_{h,d}$. Then
\begin{equation}\label{linearresponse}
J(u_h)=\bm{g}_{h,d}^T\bm{u}_h=\bm{c}_S^TA_S^+\bm{b}_S.
\end{equation}
Following \cite[Section~5.3.2]{JinLiuMaYu2026BPXrevised}, assume direct access to
\[
O_{Sc}\ket0=\ket{\bm{c}_S}=\frac{\bm{c}_S}{\|\bm{c}_S\|_2},
\]
with $\|\bm{c}_S\|_2$ supplied. Construction of this oracle is outside the query model. As in \eqref{preconditionedinputnorm},
\[
\|\bm{c}_S\|_2\le\frac1{\sqrt3}\|\Pi_hg\|_{L^2}\le\frac1{\sqrt3}\|g\|_{L^2}.
\]
If either preconditioned vector is zero, $J(u_h)=0$ and no quantum estimation of the discrete functional is needed. Otherwise, we estimate \eqref{linearresponse} directly, without applying $\mathcal E$ or amplifying a solution branch.

\begin{theorem}[Functional error and query complexity]\label{thm:linearfunctional}
Let $u\in H^2(\Omega)$ be the solution of \eqref{poisson}, let $g\in L^2(\Omega)$, and assume the stated preconditioned input access. Given $\varepsilon>0$, choose the mesh so that
\begin{equation}\label{linearfeaccuracy}
3h^2\|u\|_{H^2}\|g\|_{L^2}\le\frac{\varepsilon}{3}.
\end{equation}
A direct matrix-element estimator returns $\widehat J$ satisfying $|\widehat J-J(u)|\le\varepsilon$ with probability at least $2/3$, using
\begin{equation}\label{linearqueries}
\mathcal{O}\!\Big(dL^2\Big[1+\frac{L\|f\|_{L^2}\|g\|_{L^2}}{\varepsilon}\Big]\log\!\Big[2+L\max\Big\{1,\frac{\|f\|_{L^2}\|g\|_{L^2}}{\varepsilon}\Big\}\Big]\Big)
\end{equation}
queries to the multilevel encodings and the two input oracles. Equivalently, suppressing the logarithmic factor, the bound is
\begin{equation}\label{linearqueriessimplified}
\widetilde{\mathcal{O}}\!\Big(dL^2+\frac{dL^3\|f\|_{L^2}\|g\|_{L^2}}{\varepsilon}\Big),\qquad L=\log_2(1/h).
\end{equation}
The second term dominates when $\varepsilon\le L\|f\|_{L^2}\|g\|_{L^2}$. No restriction relating $\varepsilon$ to $\|u\|_{L^2}$ is required.
\end{theorem}
\begin{proof}
\noindent\textbf{Step 1: Discretization and inverse errors.}
Lemma~\ref{lem:feerror} and \eqref{linearfeaccuracy} give
\[
|J(u)-J(u_h)|\le\|g\|_{L^2}\|u-u_h\|_{L^2}
\le3h^2\|u\|_{H^2}\|g\|_{L^2}\le\frac{\varepsilon}{3}.
\]
If $\|\bm{b}_S\|_2\|\bm{c}_S\|_2=0$, then $J(u_h)=0$, and returning $\widehat J=0$ already satisfies the required accuracy. Otherwise, choose
\begin{equation}\label{linearinverseaccuracy}
\delta=\min\Big\{1,\frac{\varepsilon}{6\|\bm{b}_S\|_2\|\bm{c}_S\|_2}\Big\},
\end{equation}
and let $J_\delta=\bm{c}_S^Tp_\delta(A_S)\bm{b}_S$. By \eqref{inverseapproximation},
\begin{equation}\label{linearalgebraicerror}
|J(u_h)-J_\delta|\le\delta\|\bm{b}_S\|_2\|\bm{c}_S\|_2\le\varepsilon/6.
\end{equation}

\noindent\textbf{Step 2: Matrix-element estimation.}
Apply $O_{Sb}$, the inverse block encoding, and $O_{Sc}^\dagger$ on the multilevel register. With all inverse-encoding flags retained, the selected matrix element is
\begin{equation}\label{linearoverlap}
\frac{J_\delta}{2L\|\bm{b}_S\|_2\|\bm{c}_S\|_2}.
\end{equation}
A Hadamard test followed by amplitude estimation \cite{Brassard2002Amplitude} estimates its real part, including the sign. Additive accuracy $\varepsilon/(6L\|\bm{b}_S\|_2\|\bm{c}_S\|_2)$ gives scalar readout error at most $\varepsilon/3$ with constant success probability, using
\[
\mathcal{O}\!\Big(1+\frac{L\|\bm{b}_S\|_2\|\bm{c}_S\|_2}{\varepsilon}\Big)
\]
calls to the inverse circuit and the input oracles.

\noindent\textbf{Step 3: Implementation error and total cost.}
Reserve $\varepsilon/6$ for circuit implementation. A sufficient operator-error bound for the matrix-element circuit is
\[
\eta\le\min\Big\{1,\frac{\varepsilon}{12L\|\bm{b}_S\|_2\|\bm{c}_S\|_2}\Big\},
\]
since rescaling multiplies its matrix-element error by $2L\|\bm{b}_S\|_2\|\bm{c}_S\|_2$. Primitive errors, including those in amplitude estimation, are allocated over the circuit calls by telescoping. The four contributions satisfy
\[
|J(u)-\widehat J|\le\frac{\varepsilon}{3}+\frac{\varepsilon}{6}+\frac{\varepsilon}{6}+\frac{\varepsilon}{3}=\varepsilon.
\]
Multiplying the readout cost by \eqref{inversequeries} gives
\[
\mathcal{O}\!\Big(dL^2\Big[1+\frac{L\|\bm{b}_S\|_2\|\bm{c}_S\|_2}{\varepsilon}\Big]\log(2+L/\delta)\Big).
\]
To express this bound in terms of continuous data, use \eqref{preconditionedinputnorm} and the corresponding estimate for $\bm{c}_S$:
\[
\|\bm{b}_S\|_2\|\bm{c}_S\|_2\le\frac13\|f\|_{L^2}\|g\|_{L^2},\qquad
\frac1\delta\le\max\Big\{1,\frac{2\|f\|_{L^2}\|g\|_{L^2}}{\varepsilon}\Big\}.
\]
Substitution proves \eqref{linearqueries}, and suppressing its logarithmic factor yields \eqref{linearqueriessimplified}.
\end{proof}

The exact solution enters the mesh condition through $\|u\|_{H^2}$, while the query bound depends on the continuous data norms $\|f\|_{L^2}$ and $\|g\|_{L^2}$. The inverse tolerance and estimation precision are chosen from the supplied norms of $\bm{b}_S$ and $\bm{c}_S$; no additional norm oracle for $f$, $g$, or $u$ is assumed. For fixed data norms, the query bound contains neither $3^{d-1}$ nor an algebraic inverse-mesh factor. The mesh enters through $L=\log_2(1/h)$, while the $\varepsilon^{-1}$ readout cost remains explicit. This conclusion assumes direct access to both preconditioned input states; their construction is not free in a complete algorithm.

The result applies to bounded $L^2$ functionals. Point evaluation, boundary fluxes, and mesh-dependent test functions require separate bounds. Data norms, input preparation, and gate overhead may depend on dimension even though $C_{\mathrm{FE}}=3$ is dimension independent. For solution-state preparation, \eqref{solutionnorm} controls the amplification cost by $\|f\|_{L^2}/\|u\|_{L^2}$. Direct scalar estimation avoids this conversion and requires no restriction relating $\varepsilon$ to the solution norm.

\section{Numerical illustration}\label{sec:numerics}

We illustrate the construction on $\Omega=(0,1)$ using
\begin{equation}\label{numericalexample}
\begin{aligned}
u(x)&=\sin(\pi x)+\tfrac14\sin(2\pi x),\\
f(x)&=(\pi^2+c)\sin(\pi x)+\tfrac14(4\pi^2+c)\sin(2\pi x).
\end{aligned}
\end{equation}
We use $c=1$ for the main experiment. Here $\|u\|_{L^2}=\sqrt{17/32}$, so the target accuracy used below satisfies $\varepsilon\le\|u\|_{L^2}/2$. The two sine modes test the relative amplitudes and the overall scale of the represented solution. Load integrals are evaluated analytically, and continuous errors are computed by 16-point Gauss--Legendre quadrature on each cell. We first report classical discretization and conditioning results, then examine the quantum circuit on a fixed mesh.

\subsection{Finite element accuracy and preconditioning}

Table~\ref{tab:classicalexperiment} reports the finite element error and matrix condition numbers for $L=2,\cdots,7$, with $h=2^{-L}$ and $N=2^L-1$. The effective condition number of $A_S$ uses only its nonzero eigenvalues. The last column compares $\|B\|\|K_{h,1}\|$, the normalization obtained by the plain product rule if each separate factor is encoded with its spectral norm, with the bound $12L$ from Theorem~\ref{thm:encoding}.

\begin{table}[htbp]
\centering
\caption{Classical matrix and finite element results for \eqref{numericalexample} with $c=1$. }
\label{tab:classicalexperiment}
\small
\begin{tabular}{rrrrrr}
\toprule
$L$ & $\|u-u_h\|_{L^2}$ & Rate & $\kappa_2(K_{h,1})$ & $\kappa_{\mathrm{eff}}(A_S)$ & $\|B\|\|K_{h,1}\|/(12L)$\\
\midrule
2 & $5.214\times10^{-2}$ & -- & 5.359 & 2.908 & 0.672 \\
3 & $1.331\times10^{-2}$ & 1.970 & 23.009 & 3.949 & 1.938 \\
4 & $3.345\times10^{-3}$ & 1.992 & 93.662 & 4.893 & 5.924 \\
5 & $8.374\times10^{-4}$ & 1.998 & 376.284 & 5.701 & 19.044 \\
6 & $2.094\times10^{-4}$ & 2.000 & 1506.776 & 6.337 & 63.555 \\
7 & $5.236\times10^{-5}$ & 2.000 & 6028.743 & 6.861 & 217.964 \\
\bottomrule
\end{tabular}
\end{table}

Figure~\ref{fig:qbpx-classical} displays the error and conditioning trends in Table~\ref{tab:classicalexperiment}. The observed convergence approaches second order. Over these meshes, $\kappa_2(K_{h,1})$ increases from $5.36$ to $6028.74$, while $\kappa_{\mathrm{eff}}(A_S)$ increases from $2.91$ to $6.86$. The last column illustrates the normalization loss discussed in \eqref{directproductnormalization}. The bound $12L$ is conservative, so the separate-factor normalization can be smaller on the coarsest mesh; its growth under refinement is the relevant distinction. Additional matrix checks with $c=0$ and $c=100$ give effective condition numbers $6.547$ and $6.495$, respectively, at $L=7$. All 18 tested pairs of $c\in\{0,1,100\}$ and $L\in\{2,\cdots,7\}$ satisfy the spectral interval in \eqref{bpxspectralbound}. These one-dimensional checks illustrate the estimates and do not test their high-dimensional scaling.

\begin{figure}[htbp]
\centering
\subfigure[Finite element convergence.\label{fig:qbpx-fe}]{%
\includegraphics[width=0.4\textwidth]{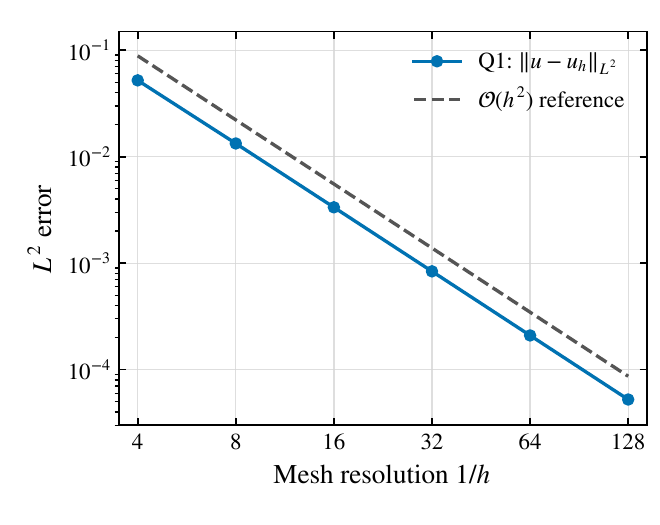}}
\subfigure[Matrix condition numbers.\label{fig:qbpx-condition}]{%
\includegraphics[width=0.4\textwidth]{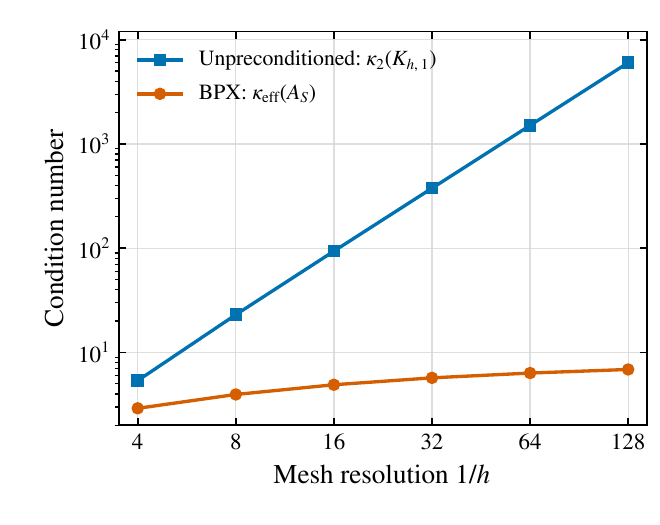}}
\caption{Finite element convergence and conditioning for \eqref{numericalexample} with $c=1$. }
\label{fig:qbpx-classical}
\end{figure}

\subsection{Circuit simulation in UnitaryLab}

The circuit experiment uses UnitaryLab 1.1.6 \cite{UnitaryLab116}, Python 3.11.15, and the NumPy CPU backend with complex128 arithmetic. We fix $L=4$ and $c=1$. The nodal system has 15 unknowns, the multilevel system has 26 coordinates, and the finest-level $L^2$-orthonormal representation has 32 coefficients. The multilevel register is padded to 32 dimensions. For numerical validation, this experiment explicitly includes a right-hand-side conversion before the inverse: it starts from $\ket{\bm{f}_h^{L^2}}$ and applies the encoding of $\mathcal E^T$, whose successful branch is $\bm{b}_S/(\alpha_{\mathcal E}\|\Pi_hf\|_{L^2})$. It then applies the inverse and the encoding of $\mathcal E$. Both value-map encodings have normalization $1/\sqrt3$ and separate success flags. This extra input conversion tests the identity $\mathcal E^T\bm{f}_h^{L^2}=\bm{b}_S$; it is not a step charged to the direct-input algorithm of Section~\ref{sec:accuracy}.

For this small reference experiment, the inverse is implemented using phase estimation \cite[Chapter~5]{Nielsen2010} with $m$ phase qubits for $\exp(2\pi\i A_S/(12L))$. A phase-register value $j$ represents $\lambda_j=12Lj/2^m$. A controlled rotation gives a success amplitude
\[
r_j=\begin{cases}
0,&\lambda_j<(2L)^{-1},\\
(2L\lambda_j)^{-1},&\lambda_j\ge(2L)^{-1}.
\end{cases}
\]
Phase estimation is then uncomputed. Let $R_m/(2L)$ denote the inverse block selected by the zero phase register and the successful rotation flag. The entire response circuit uses $m+8$ qubits. Its selected spatial branch is multiplied by the known scale $2L\alpha_{\mathcal E}^2\|\Pi_hf\|_{L^2}$ to obtain $E_L\bm{u}_h^q=\mathcal E R_m\bm{b}_S$. Thus the represented function agrees with that obtained from a directly supplied $\ket{\bm{b}_S}$ and the same inverse approximation. The tabulated qubit counts and joint success probabilities include the additional input conversion; they are not those of the direct-input circuit in \eqref{solutionbranch}. The exact solution norm is not used in the rescaling.

The controlled propagators and unitary dilations of the value maps are supplied as dense matrix gates. The experiment therefore tests the inverse and $L^2$ solution representation in a finite-dimensional oracle model. The local block-encoding circuits and QSVT inverse used in the complexity analysis are not implemented in this simulation. The classical finite element solution provides the reference for evaluating errors.

\begin{table}[htbp]
\centering
\caption{UnitaryLab response-circuit results at $L=4$ and $c=1$. Here $\delta_m=\|R_m-A_S^+\|$, and the last column is the joint success probability before amplification.}
\label{tab:quantumexperiment}
\small
\begin{tabular}{rrrrrr}
\toprule
$m$ & Qubits & $\delta_m$ & $\|u_h-u_h^q\|_{L^2}$ & $\|u-u_h^q\|_{L^2}$ & Success probability\\
\midrule
6 & 14 & $1.500\times10^{-2}$ & $2.037\times10^{-2}$ & $2.281\times10^{-2}$ & $6.363\times10^{-4}$ \\
8 & 16 & $1.951\times10^{-3}$ & $3.830\times10^{-3}$ & $3.172\times10^{-3}$ & $6.797\times10^{-4}$ \\
10 & 18 & $1.735\times10^{-3}$ & $3.334\times10^{-3}$ & $3.061\times10^{-3}$ & $6.787\times10^{-4}$ \\
12 & 20 & $4.037\times10^{-4}$ & $5.990\times10^{-4}$ & $2.889\times10^{-3}$ & $6.739\times10^{-4}$ \\
\bottomrule
\end{tabular}
\end{table}

Table~\ref{tab:quantumexperiment} separates the circuit error from the total continuous error. At $m=12$, these errors are $5.9900\times10^{-4}$ and $2.8893\times10^{-3}$, respectively, while the finite element error is $3.3451\times10^{-3}$. Since $u-u_h^q=(u-u_h)+(u_h-u_h^q)$, partial cancellation can make the total error smaller than the finite element error. For the target $\varepsilon=1.2\times10^{-2}$, the measured discretization and circuit contributions are each below $\varepsilon/3$ for $m\ge8$. All four response vectors agree with an independent evaluation of the phase-estimation spectral filter to within $2\times10^{-14}$ in Euclidean norm.

\begin{figure}[htbp]
\centering
\subfigure[Pointwise absolute errors.\label{fig:qbpx-pointwise}]{%
\includegraphics[width=0.4\textwidth]{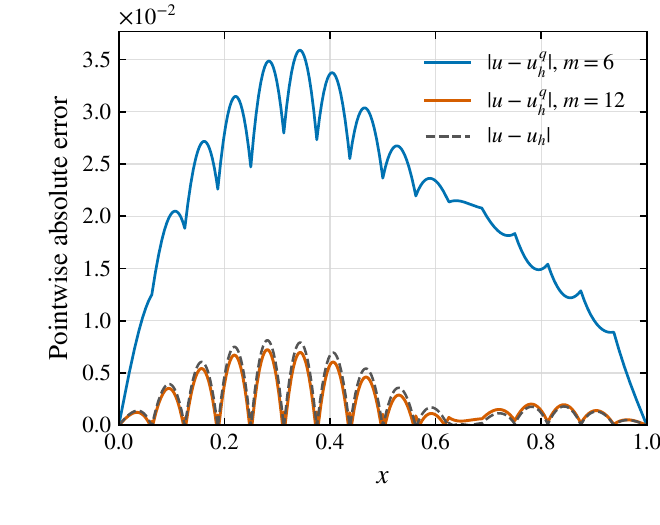}} 
\subfigure[Continuous errors versus $m$.\label{fig:qbpx-circuit}]{%
\includegraphics[width=0.4\textwidth]{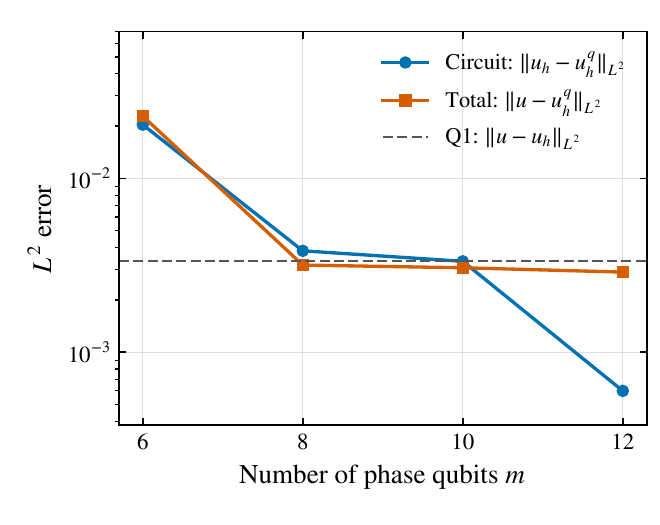}}
\caption{UnitaryLab solution errors at $L=4$ and $c=1$. Panel (a) compares pointwise errors at $m=6$ and $m=12$. Panel (b) shows the circuit and total continuous errors as the phase-register size $m$ increases at fixed $h=1/16$. Dashed lines indicate the Q1 finite element errors; the horizontal reference in (b) is $\|u-u_h\|_{L^2}\approx3.3451\times10^{-3}$, which is independent of $m$.}
\label{fig:qbpx-quantum}
\end{figure}

Figure~\ref{fig:qbpx-quantum} shows how increasing the phase-register size reduces the circuit contribution. At $m=12$, the pointwise errors are close to those of the finite element approximation. The total $L^2$ error is then dominated by spatial discretization; its position slightly below the finite element reference reflects the partial cancellation described above.

To test direct functional estimation, take $g(x)=\sin(\pi x)+\tfrac12\sin(2\pi x)$, so that $J(u)=(g,u)=9/16$. Let $\bm{g}_h^{L^2}$ be the orthonormal coefficients of $\Pi_hg$. At $m=12$, a separate 21-qubit Hadamard-test circuit evaluates $(\bm{g}_h^{L^2})^T\mathcal E R_m\mathcal E^T\bm{f}_h^{L^2}=\bm{c}_S^TR_m\bm{b}_S$, retaining all encoding flags. As in the solution experiment, it explicitly includes the input and test-vector conversions that are absorbed into the supplied oracles in Section~\ref{sec:accuracy}. Rescaling its simulated expectation gives
\[
\widehat J=0.5607326198,\qquad
|\widehat J-J(u)|=1.7674\times10^{-3}.
\]
The corresponding finite element value is $J(u_h)=0.5602631232$. The scalar is obtained directly from the response circuit, without first preparing the normalized solution state. Probabilities and expectations are evaluated from simulated statevectors; finite-shot sampling, amplitude amplification, and amplitude estimation are not included. These experiments verify the coefficient scaling and response formulas but do not establish the asymptotic query bounds numerically.


\section{Comparison and concluding discussion}\label{sec:discussion}

This article explains quantum BPX preconditioning for Q1 reaction--diffusion discretizations within the framework of \cite{DeimlPeterseim2025QFEM,JinLiuMaYu2026BPXrevised}. The construction starts from one-dimensional Gram factors, extends them by tensor products, and uses isometric transfers to assemble the multilevel energy map. For $S_\ell=\omega_\ell P_\ell M_\ell^{-1/2}$, the nonzero spectral interval is $[1/L,12dL]$ and the encoding normalization of $A_S$ is $12dL$, uniformly in $c\ge0$. The Poisson problem is obtained at $c=0$, and each shifted system in a fractional rational approximation has the same reaction--diffusion structure.

Normalizing the one-dimensional value maps before forming tensor products avoids the explicit factor $3^{d-1}$. With the preconditioned right-hand side supplied directly, the inverse and preparation of the multilevel variable state have no algebraic inverse-mesh factor. For the solution state in the local $L^2$-orthonormal basis, the map $\mathcal E$ introduces the amplification ratio $\|\bm{b}_S\|_2/\|u_h\|_{L^2}$. For $0<\varepsilon\le\|u\|_{L^2}/2$, this ratio is bounded by $6\|f\|_{L^2}/(5\sqrt3\|u\|_{L^2})$. The resulting bound \eqref{queriesdata} contains no algebraic inverse-mesh factor when $\|f\|_{L^2}/\|u\|_{L^2}$ is fixed. Its data dependence is essential; increasing the magnitude of the solution alone does not improve the amplification ratio. With the preconditioned test state also supplied, direct estimation of a bounded $L^2$ functional avoids the solution conversion, requires no restriction on the accuracy relative to the solution norm, and satisfies \eqref{linearqueriessimplified}.

For comparison, a classical tensor sine solve with an assembled load vector takes $\mathcal{O}(dLN)$ arithmetic operations, where $N=(2^L-1)^d$, and returns all nodal coefficients. The quantum bounds count queries to supplied preconditioned-input oracles and matrix encodings, and produce either a normalized local $L^2$-orthonormal coefficient state or a prescribed scalar estimate. The solution-state bound uses this representation of the same finite element function; further conversion to a state in the original nodal basis would require a separate cost analysis. Constructing these oracles, elementary gate costs, and classical readout must therefore be included when comparing complete algorithms. 

\section*{Acknowledgments}

Y. Yang was supported by NSFC grant (No.\ 12571469), Scientific Research Innovation Capability Support Project for Young Faculty of China (No.\ SRICSPYF-BS2025132), the Project of Scientific Research Fund of the Hunan Provincial Science and Technology Department (No.\ 2024JJ1008), the 111 Project (No. D23017), Major Scientific and Technological Innovation Platform Project of Hunan Province (No.\ 2024JC1003), and Program for Science and Technology Innovative Research Team in Higher Educational Institutions of Hunan Province of China.
Y. Yu was supported by NSFC grant (No.\ 12301561), the Key Project of Scientific Research Project of Hunan Provincial Department of Education (No.\ 24A0100), the Science and Technology Innovation Program of Hunan Province (No.\ 2025RC3150) and the general program of Hunan Provincial Natural Science Foundation (No.\ 2026JJ50003), and partially supported by NSFC grant No. 12341104.
L. Zhang was supported by the Hunan Provincial Graduate Student Research and Innovation Project (No.\ CX20250933) and the Xiangtan University Graduate Student Research and Innovation Project (No.\ XDCX2025Y188).

\appendix
\section{Proof of Lemma~\ref{lem:feerror}}\label{app:feconstant}

We track the constants in the interpolation and inverse estimates used in the standard finite element error argument; see \cite[Sections~4.4--4.5]{Brenner2008}. Let $Q$ be the $L^2(0,1)$ projection onto continuous piecewise linear functions with zero endpoint values on a uniform mesh of width $h$. With $\gamma=1+\sqrt{12}/\pi$, these estimates give
\[
\begin{aligned}
\|v-Qv\|_{L^2(0,1)}&\le\frac h\pi\|v'\|_{L^2(0,1)},\\
\|(Qv)'\|_{L^2(0,1)}&\le\gamma\|v'\|_{L^2(0,1)},\\
\|(v-Qv)'\|_{L^2(0,1)}&\le\frac{\gamma h}{\pi}\|v''\|_{L^2(0,1)}.
\end{aligned}
\]
The first two inequalities hold for $v\in H_0^1(0,1)$, and the third additionally requires $v\in H^2(0,1)$. The first bound also follows from the degree-one Ritz estimate in \cite[Theorem~4]{SandeManniSpeleers2020} and the best-approximation property of $Q$. Indeed, for the nodal interpolant $Iv$, cellwise Poincar\'e inequalities give $\|v-Iv\|\le(h/\pi)\|(v-Iv)'\|$, $\|(v-Iv)'\|\le(h/\pi)\|v''\|$, and $\|(Iv)'\|\le\|v'\|$. Combining these with $QI=I$, the contraction property of $Q$, and the inverse bound $\|w_h'\|\le\sqrt{12}\,h^{-1}\|w_h\|$ from \cite[Remark~12]{SandeManniSpeleers2020} yields the stated bounds.

Let $Q^{(q)}$ act as $Q$ in coordinate $q$, so that $\Pi_h=\prod_{q=1}^dQ^{(q)}$. This coordinatewise construction is the tensor-product projection described in \cite[Section~6]{SandeManniSpeleers2020}. For each $p$, use
\[
v-\Pi_hv=(I-Q^{(p)})v+Q^{(p)}\Big(I-\prod_{q\ne p}Q^{(q)}\Big)v.
\]
The preceding one-dimensional bounds and orthogonality of successive projection errors in the other coordinates imply
\[
\|\partial_p(v-\Pi_hv)\|_{L^2}
\le\frac{\gamma h}{\pi}\Big(\|\partial_{pp}v\|_{L^2}
+\Bigl(\sum_{q\ne p}\|\partial_{pq}v\|_{L^2}^2\Bigr)^{1/2}\Big).
\]
Squaring and summing over $p$ therefore gives
\[
\|\nabla(v-\Pi_hv)\|_{L^2}
\le\frac{\sqrt2\gamma}{\pi}h\|D^2v\|_{L^2},
\qquad
\|D^2v\|_{L^2}^2:=\sum_{p,q=1}^d\|\partial_{pq}v\|_{L^2}^2.
\]

Let $w_h\in V_h$ be the Poisson Ritz projection of $u$. Galerkin orthogonality gives the same bound for $\|\nabla(u-w_h)\|_{L^2}$. For the dual solution $-\Delta z=u-w_h$ with zero boundary values on the cube, the sine expansion gives $\|D^2z\|_{L^2}=\|u-w_h\|_{L^2}$. The Aubin--Nitsche argument \cite[Chapter~5]{Brenner2008}, testing with $z-\Pi_hz$ and applying the approximation bound to both $u$ and $z$, yields
\[
\|u-w_h\|_{L^2}\le\frac{2\gamma^2}{\pi^2}h^2\|D^2u\|_{L^2}.
\]
To compare with the reaction--diffusion approximation, set $e_h=u_h-w_h$. The discrete equations give
\[
a(e_h,v_h)=c(u-w_h,v_h),\qquad v_h\in V_h.
\]
For $c>0$, testing with $e_h$ gives $\|e_h\|_{L^2}\le\|u-w_h\|_{L^2}$; for $c=0$, $e_h=0$. Since $\|D^2u\|_{L^2}\le\sqrt2\|u\|_{H^2}$, it follows that
\[
\|u-u_h\|_{L^2}
\le\frac{4\sqrt2}{\pi^2}\Big(1+\frac{\sqrt{12}}{\pi}\Big)^2h^2\|u\|_{H^2}
<3h^2\|u\|_{H^2}
\]
for $u\ne0$, with the zero solution immediate. This proves the stated uniform constant.

\renewcommand{\baselinestretch}{1.15}\small
\bibliographystyle{unsrt}
\bibliography{Refs}
\end{document}